\documentclass[a4paper,11pt]{article}   
\usepackage[utf8]{inputenc}    
\usepackage{graphicx}
\usepackage{subcaption} 
\usepackage{float}   
\usepackage{fancybox}		  
\usepackage{makeidx} 
\usepackage{verbatim}
\usepackage{listings} % package for a colored Matlab code            
\usepackage{fancyvrb}
\usepackage{amsfonts}
\usepackage{color}
\usepackage{colortbl}
\usepackage{color}
\usepackage{dsfont}
\usepackage{epsfig}
\usepackage{verbatim}
\usepackage{listings}
\usepackage{bbm}
\usepackage{tikz}
\definecolor{dgreen}{rgb}{0.0, 0.5, 0.0}
\definecolor{byzantium}{rgb}{0.44, 0.16, 0.39}

\usepackage{amssymb}
\usepackage{amsbsy}
\usepackage{amsmath}
\usepackage{enumerate}
\usepackage{stmaryrd}
\usepackage{mathtools}
\usepackage{tikz}
\usepackage{bbm}

\usepackage{fancyhdr}

\usepackage{amsthm}

\newtheorem{prop}{Proposition}%[section]
\newtheorem{theorem}{Theorem}%[section]
\newtheorem{hyp}{Hypothesis}%[H\arabic{hyp}]
\newtheorem{lemma}{Lemma}%[section]
\newtheorem{rem}{Remark}[section]%[P\arabic{prob}]
\newtheorem{definition}{Definition}%[section]
\usepackage[margin=1in]{geometry}
\usepackage{hyperref}
\hypersetup{                    % parametrage des hyperliens
    colorlinks=true,                % colorise les liens
    breaklinks=true,                % permet les retours Ã  la ligne pour les liens trop longs
    urlcolor= blue,                 % couleur des hyperliens
    linkcolor= red,                % couleur des liens internes aux documents (index, figures, tableaux, equations,...)
    citecolor= blue                % couleur des liens vers les references bibliographiques
    }

\makeatletter
\DeclareRobustCommand\widecheck[1]{{\mathpalette\@widecheck{#1}}}
\def\@widecheck#1#2{%
    \setbox\z@\hbox{\m@th$#1#2$}%
    \setbox\tw@\hbox{\m@th$#1%
       \widehat{%
          \vrule\@width\z@\@height\ht\z@
          \vrule\@height\z@\@width\wd\z@}$}%
    \dp\tw@-\ht\z@
    \@tempdima\ht\z@ \advance\@tempdima2\ht\tw@ \divide\@tempdima\thr@@
    \setbox\tw@\hbox{%
       \raise\@tempdima\hbox{\scalebox{1}[-1]{\lower\@tempdima\box
\tw@}}}%
    {\ooalign{\box\tw@ \cr \box\z@}}}
\makeatother

\def\R {\mathbb{R}}

\def\E {\mathbb{E}}
\def\F {\mathcal{F}}

\newcommand{\dsp}{\displaystyle}

\newcommand{\sumj}{\sum_{j=1}^N}

\newcommand{\supp}{\mathrm{supp}}

\newcommand{\schema}[1]{\b{\sc #1}}

\newcommand{\ty}{\tilde{y}}

\DeclareMathOperator*{\esssup}{ess\,sup}

\newcommand{\xz}{{\xi,\zeta}}
\newcommand{\tw}{\tilde{w}}
\newcommand{\txi}{\tilde{\xi}}
\newcommand{\tz}{\tilde{\zeta}}
\newcommand{\bxi}{\bar X_i^N}
\newcommand{\bxj}{\bar X_j^N}
\newcommand{\bwij}{\bar W_{ij}^N}
\newcommand{\lij}{\bar \lambda_t^{N,i,j}}
\newcommand{\lijun}{\bar \lambda_t^{N,i_1,j_1}}
\newcommand{\tx}{{\tilde x}}
\newcommand{\xztt}{{\txi,\tz}}
\newcommand{\intt}{\int_{I \times \R^d \times \R^d \times \R}}
\newcommand{\inttt}{\int_{I^2 \times \R^d \times \R^d \times \R}}

\DeclareMathOperator{\divop}{div}

\newcommand{\croch}[1]{\left[ #1 \right]} 
\newcommand{\pare}[1]{\left( #1 \right)}

\begin{document}

\title{Mean-field limits  for interacting particle systems on general adaptive dynamical networks}

\author{
Nathalie Ayi\thanks{Sorbonne Universit\'e, Universit\'e Paris Cit\'e, CNRS, Inria, Laboratoire Jacques-Louis Lions (LJLL), Intitut Universitaire de France, F-75005 Paris, France} }

%\date{}
\maketitle
%\tableofcontents

\bibliographystyle{abbrv}

%%%%%%%%%%

\abstract{We study the large-population limit of interacting particle systems evolving on adaptive dynamical networks, motivated in particular by models of opinion dynamics. In such systems, agents interact through weighted graphs whose structure evolves over time in a coupled manner with the agents’ states, leading to non-exchangeable dynamics.  In the dense-graph regime, we show that the asymptotic behavior is described by a Vlasov-type equation posed on an extended phase space that includes both the agents’ states and identities and the evolving interaction weights.
We establish this limiting equation through two complementary approaches. The first follows the  mean-field methodology in the spirit of Sznitman \cite{Snitzman}. 
In this framework, we impose the additional assumption that the weight dynamics is independent of one of  the agent’s states, an assumption that remains well motivated from a modeling perspective and allows for a direct derivation of the mean-field limit. The second approach is based on the graph limit   framework and is formulated in a deterministic setting. This perspective makes it possible to remove the aforementioned restriction on the weight dynamics and to handle more general interaction structures.
Our analysis includes well-posedness and stability results for the limiting Vlasov-type equation, as well as quantitative estimates ensuring the propagation of independence. We further clarify the relationship between the continuum (graph limit) formulation and the mean-field limit, thereby providing a unified description of the asymptotic dynamics of interacting particle systems on adaptive dynamical networks.}

\section{Introduction}

Interacting particle systems are powerful mathematical models used to describe collective dynamics. They naturally arise in various contexts, such as modeling the evolution of opinions within a population \cite{HK}, or the flocking and swarming behavior observed in groups of birds or schools of fish \cite{Aoki82, Ballerini08, Lopez12}. If we interpret each   individual within the population as a ``particle'', the early studies of such models mostly focused on indistinguishable particles, i.e. situations where exchanging the roles of two particles does not alter the dynamics of the system.
In this setting, the mean-field approach has been extensively developed in the literature (see  \cite{braun1977,Dobrushin79,Snitzman} for instance). This framework involves a change of perspective: instead of tracking the precise trajectory of each individual over time, we study the statistical distribution of positions. The main object of interest becomes a measure $\mu_t(x)$, representing the probability of finding an agent with opinion $x$ at time $t$ in the context of opinion dynamics.

However, there are situations, particularly in opinion dynamics, where the identity of each particle matters. Indeed, human interactions depend on social relationships, which determine both whom we interact with and how much influence others have on us. The mathematical structure that captures this idea is a graph. A typical model is given by the following equation:

\begin{equation}\label{eq:non-exch_part_syst}
\displaystyle \frac{d}{dt} x_i^N(t) = \frac{1}{N} \sumj w_{ij}\phi(x_i^N(t),x_j^N(t)), \, \text{ for } i=1, \dots, N,
\end{equation}
 where $t \geq 0$ is the time, $x_i^N\in\mathbb{R}^d$ with $d \geq 1$, is the state variable, representing for instance the opinion of agent $i$ and $\phi$ is the interaction function. We can interpret the previous equations as a system of differential equations on the weighted graph $$G_N = <V(G_N),E(G_N), w^N>$$ whose vertex set is $V(G_N) = \{1, \dots, N\}$, edge set is $E(G_N) = \{1, \dots,N \}^2$  and the weights are $w^N=(w_{ij}^N)_{1 \leq i,j \leq N}$ with $w_{ij} \in \R$ for all $i,j \in \{1\dots,N \}$. 
 
A natural question is then whether classical mean-field limit results can be extended to this more complex framework, where the particle system becomes non-exchangeable. Recent progress in graph theory, building on the seminal work of Lovász \cite{Lovasz12}, has yielded positive answers to this question. These results are supported by a substantial and growing literature (see, e.g., \cite{KaliuzhnyiMedvedev18,ChibaMedvedev19,Kuehn20,JabinPoyatoSoler21,KuehnXu22}), as well as by the recent survey \cite{AyiPouradierDuteil24} and the lecture notes on the subject \cite{Ayi2026}. Depending on whether the graph is dense, sparse, or of intermediate density, different limiting objects arise to describe the asymptotic behavior of a convergent sequence of graphs and  a limit equation on $\mu_t^\xi(x)$, representing the probability of finding an agent with identity $\xi \in [0,1]$ and opinion $x$ at time $t$, can be obtained.

 In the dense case, which is the one similar to framework we will adopt on this paper, we exploit the convergence of a sequence of graphs toward a graphon (short for \textit{graph function}). The emergence of graphons as natural limits of large graphs can be intuitively understood by considering the visual convergence of their pixel matrices, a pixel version of the adjacency matrices where 1’s correspond to black squares and 0’s to white ones. For further details, we refer an interested reader to the lecture notes \cite{Ayi2026}.

In this paper, the framework we will focus on is that of interacting particle systems on adaptive dynamical networks, i.e. graphs whose structure evolves dynamically and in a coupled manner with the node states. Returning to our example of opinion dynamics, individuals’ opinions are influenced by their social contacts, but these relationships themselves may evolve: one may stop interacting with, or stop being influenced by, people holding very different opinions, while forming new connections with like-minded individuals. This adaptive setting thus provides a more realistic modeling framework than a purely static graph, encompassing a wide range of real-life applications   (see \cite{rev_kuehn} for a review on the matter).

In this setting, we may begin by considering a particular example of interacting particle systems on an adaptive dynamical network, namely the one introduced in \cite{McQuadePiccoliPouradierDuteil19} and whose large-population limit was studied in \cite{AyiPouradierDuteil21}. The system can be seen as an extended model of an exchangeable interacting particle system where, in addition to holding an opinion, each agent is assigned an influence weight that evolves over time in a coupled way with the opinions. The weight can be interpreted as a measure of charisma or popularity and the system writes as follows:
\begin{equation}\label{eq:weightvaryingdynamics}
\begin{cases}
\displaystyle
\frac{d}{dt}x_i^N(t) = \frac{1}{N}\sum_{j=1}^N m_j^N(t)\phi(x_i^N(t), x_j^N(t)) \\
\displaystyle \frac{d}{dt}m_i^N(t) = \psi_i(x^N(t),m^N(t))
\end{cases}
\end{equation}
\noindent with $x^N=(x_1^N, \dots, x_N^N)$, $m^N=(m_1^N, \dots, m_N^N)$ and where $x_i^N \in\mathbb{R}^d$ is, as before,  the opinion of agent $i$,  $m_i^N \in \R^+$ the agent's weight. The distinctive feature of this graph is that all the weight of the edges emanating from a given vertex are identical. It should be noted, however, that this restriction results from the specific construction of the model and is not required to obtain the large-population limit. Indeed, in \cite{Throm_2024}, under the same type of regularity assumptions, the author removed this restriction to study the system 
\begin{equation}\label{eq:generalweightvaryingdynamics}
\begin{cases}
\displaystyle
\frac{d}{dt}x_i^N(t) = \frac{1}{N}\sum_{j=1}^N w_{ij}^N(t)\phi(x_i^N(t),x_j^N(t)) \\
\displaystyle \frac{d}{dt} w_{ij}^N(t) = \psi_{ij}(x^N(t),w^N(t)).
\end{cases}
\end{equation}
In the aforementioned papers, for non-exchangeable particle systems in this context of adaptive dynamical networks, only what is known as the \textit{graph limit} had been obtained so far. This approach consists in deriving an integro-differential Euler-type equation for the quantity $x(t,\xi)$, which represents the opinion of agent with identity $\xi \in [0,1]$ at time $t$. Conceptually, this corresponds to replacing discrete indices by continuous ones, sums by integrals, and a convergent sequence of weighted graphs by a graphon. For instance, in the static case, the limiting equation associated with system \eqref{eq:non-exch_part_syst} takes the form
\begin{equation}\label{eq:GL_nonexchangeable}
 \partial_t x(t,\xi) = \int_{[0,1]} w(\xi,\zeta) \phi(x(t,\xi),x(t,\zeta)) d\zeta 
\end{equation}
for $t \geq 0$, $\xi\in [0,1]$.
%(Equation)
This approach, also commonly referred to as   \textit{continuum limit} can be extended in the framework of \cite{AyiPouradierDuteil21} and the more general case \cite{Throm_2024} where it leads for the latter to convergence to 
\begin{equation}
\left\{\begin{array}{l}
\displaystyle \partial_t x(t,\xi) = \int_I w(t,\xi,\zeta)\phi(x(t,\xi),x(t,\zeta)) d\zeta \\
\partial_t w(t,\xi,\zeta) = \psi(\xi, \zeta,x(t,\cdot),w(t,\cdot, \cdot)).
\end{array}\right.
\label{eq:GL_Throm}
\end{equation} 
The first mean-field limit result, obtained in the slightly different context of digraph measures, which makes it possible to handle sparse graphs, appeared in \cite{gkogkas2023mean}. In that work, the author established the mean-field limit for systems of the form 
\begin{equation}\label{eq:weightvaryingdynamics2}
\begin{cases}
\displaystyle
\frac{d}{dt}x_i^N(t) = \frac{1}{N}\sum_{j=1}^N w_{ij}^N(t)\phi(x_i^N(t),x_j^N(t)) \\
\displaystyle \frac{d}{dt}w_{ij}^N(t) = - \varepsilon w_{ij}^N + \varepsilon H(x_i^N(t),x_j^N(t))
\end{cases}
\end{equation}
It is important to emphasize that the decoupling between the weights and the opinions in the weight dynamics is essential. More precisely, they need to be able to express the evolution of the weights in terms of their initial values and the agents’ opinions, and then substitute this expression back into the first equation. This can be achieved by applying Duhamel’s formula:
\begin{equation}
w_{ij}^N(t) = w_{ij}^{N,0}(t) e^{-\varepsilon t} + \int_0^t \varepsilon H(x_i^N(s),x_j^N(s)) e^{- \varepsilon(t-s)} ds 
\end{equation}
and leads to study the equation
\begin{equation}\label{eq:weightvaryingdynamics2bis}
\frac{d}{dt}x_i^N(t) = \frac{1}{N}\sum_{j=1}^N\left( w_{ij}^{N,0}(t) e^{-\varepsilon t} + \int_0^t \varepsilon H(x_i^N(s),x_j^N(s)) e^{- \varepsilon(t-s)} ds\right) \phi(x_i^N(t),x_j^N(t)) 
\end{equation}
More recently, in \cite{Throm2025}, they extended this result by obtaining the mean-field limit of the interacting particle system 
\begin{equation*}\label{eq:kuramoto}
\displaystyle \frac{d}{dt} x_i =  \frac{1}{N}  \sum_{j=1}^N  K_t\left(w_{ij}^{N,0}, x_i^N |_{[0,t]}, x_j^N |_{[0,t]}  \right) \quad \text{ for all } i\in \{1, \dots, N \} 
\end{equation*}

In a different direction, a related result on the propagation of independence  has been obtained in \cite{zhou2025} for the complex interacting particle system
\begin{equation*}\label{eq:kuramoto-gen}
\begin{cases}
\displaystyle d x_i = \mu(x_i) dt + \nu d{B}t^i + \frac{1}{N} \sum_{j=1}^N w_{ij} \phi\left( x_i, x_j\right) dt,\\
\displaystyle d w_{ij} = \alpha(x_i,x_j) w_{ij} dt + H( x_i, x_j) dt,
\end{cases}
\end{equation*}
His result does not constitute a mean-field limit in the strict sense, since it does not yield a Vlasov-type equation. Nevertheless, he proves a propagation of independence property: by introducing a corresponding McKean-Vlasov SDE system, he is able to control the error incurred when approximating the particle system by this alternative one. This type of control typically represents, as will be seen throughout this paper, the first crucial step toward rigorously establishing a mean-field limit.

In this paper, we aim to derive the mean-field limit of \eqref{eq:generalweightvaryingdynamics} for mean-field type interactions for the weight dynamics. More precisely, we are interested into the following adaptive dynamical network system 
 \begin{align} \label{eq:particle_syst}
\left\{ \begin{aligned}
 \frac{d}{dt} X_i^N & =  \frac{1}{N} \sum_{j=1}^N W_{ij}^N \phi(t,X_i^N,X_j^N)\\
  \frac{d}{dt}  W_{ij}^N  & =  \frac{1}{N^2} \sum_{i_1=1}^N  \sum_{j_1=1}^N  \Lambda_{ij}(X_i^N,X_j^N,W_{ij}^N,X_{i_1}^N,X_{j_1}^N,W_{{i_1}{j_1}}^N)
 \end{aligned} \right.
 \end{align}
where 
 \begin{equation}
 \Lambda_{ij}(x,y,w, \tx, \ty, \tw) :=   \int_{I_i^N\times I_j^N} N^2 \Lambda(\bar \xi,\bar \zeta,x, y,w, \tx, \ty, \tw)   d \bar \xi d \bar \zeta,
 \end{equation}
 with, for every $ k \in \left\{1, \dots, N\right\}, \,  I_k^N :=\left[\frac{k-1}{N},\frac{k}{N}\right)$, and where the interaction function $\phi: \R_+  \times  \R^d \times \R^d  \to \R $ and the edge dynamics $\Lambda: [0,1]^2 \times  \R^d \times \R^d \times \R  \times   \R^d \times \R^d \times \R \to \R$ are functions whose properties will be precised later. 
\begin{rem}
This choice of weight dynamics situates our work within the same framework as previous studies on large-population limits for adaptive dynamical networks \cite{AyiPouradierDuteil21,Throm_2024}, which have predominantly focused on the continuum limit, while here we will extend these results by presenting a rigorous mean-field analysis. 
\end{rem}

\begin{rem}From a modeling point of view, a recent work on the dynamics of complex systems \cite{Ana} has highlighted the limitations of models based exclusively on pairwise interactions between nodes. Indeed, many real-world systems are characterized by higher-order interactions involving multiple entities simultaneously, whose natural structure is more accurately described by topological objects such as simplicial complexes. In this setting, interactions are no longer restricted to vertices but may also involve edges, triangles, or higher-dimensional simplices, thereby capturing collective mechanisms that cannot be reduced to purely binary interactions between nodes. In particular, our weight dynamics  $\Lambda_{ij}$  naturally fits into this framework, as it depends not only on the states $X_i^N$ and $X_j^N$ of the nodes it connects and on the edge weight $W_{ij}^N$, but also on the states $X_{i_1}^N$, $X_{j_1}^N$ and weight $W_{{i_1}{j_1}}^N$ of neighboring vertices and edges, effectively incorporating higher-order, edge–edge interactions.
 Such edge–edge interactions can in turn give rise to novel collective phenomena, including discontinuous phase transitions that are absent in classical node-based models.
\end{rem}

In the whole paper, we will assume that for $i=j$, $W_{ij}(t)=0$ for all $t \geq 0$. This assumption reflects the absence of loops in the graph and models the fact that self-interactions are not taken into account.

 In this paper, we obtain a rigorous proof for the mean-field limit. More precisely, we prove the convergence as $N$ goes to $\infty$ of the particle system to the associated Vlasov-type equation :
\begin{align}\label{eq:vlasov-equation}
\left\{ \begin{aligned}
&\partial_t \mu_t^\xz (x,y,w)+\divop_{x,y,w}\left(F_\Lambda[\mu_t](\xi,\zeta,x,y,w)\,\mu_t^\xz(x,y,w)\right)=0,\quad t \in [0,T],\,x, \,y\in \mathbb{R}^d,\,w \in \R,\,\xi, \zeta\in [0,1],\\
&\mu_{t=0}^\xz=\mu_0^\xz.
\end{aligned} \right.
\end{align}
 with \begin{equation} \label{eq:vlasov-force}
F_\Lambda[\mu_t](\xi,\zeta,x,y,w) = \left(  \begin{array}{c}
\displaystyle  \intt \tw \phi(t,x,\ty)  \mu_t^{\xi,\tz} (d\tx,d\ty, d\tw) d\tz  \\ 
\displaystyle  \intt \tw \phi(t,y,\ty) \mu_t^{\zeta,\tz} (d\tx,d\ty, d\tw) d\tz  \\ 
\displaystyle \inttt \Lambda( \xi, \zeta, x,y,w, \tx, \ty, \tw) \mu_t^\xztt(d\tx,d\ty,d\tw)d \txi d \tz
\end{array} \right)
\end{equation} 
and  where $T>0$  denotes a constant fixed throughout the paper.
Here, the limit measure $\mu_t^\xz (x,y,w)$ represents the probability measure of finding agents of respective identities $\xi$ and $\zeta$ with respective opinions $x$ and $y$ such that the weight from  agent $\xi$ to agent $\zeta$ is $w$. 
A fundamental ingredient of our approach is that the limiting dynamics are formulated on an extended phase space, which includes both the particle identities and states and the adaptive interaction weights. This extension is not merely a technical convenience but a structural necessity: in adaptive networks, the evolution of the weights is intrinsically coupled to the microscopic states, and any formulation restricted to the state variables alone leads to an open, non-closed system.

To the best of our knowledge, this work provides the first rigorous derivation of a closed Vlasov-type equation for such a general class of adaptive interacting particle systems. The key to this closure lies precisely in the use of the extended phase space, which offers the minimal framework in which the macroscopic dynamics can be expressed in a self-consistent and closed form.

{
An important aspect of our analysis is that the derivation of this Vlasov-type equation in an extended phase space can be achieved through two distinct approaches. The first follows Sznitman’s mean-field methodology under a framework where $\Lambda_{ij}$   is required to be independent of  $X_i$. As will be discussed, this assumption remains well justified from a modeling viewpoint. The second approach builds on the graph-limit perspective, where this restriction is no longer necessary.}

The paper is organized as follows.
 In Section \ref{sec:model}, we first introduce the general setting and assumptions of the model, then describe the functional spaces used throughout the analysis, before stating the main theorems of the paper.  Section \ref{sec:well-posedness} is devoted to the well-posedness of the limiting Vlasov-type equation.
In Section \ref{sec:stability}, we prove stability estimates for this  equation, which play a crucial role in the subsequent analysis.
Section \ref{sec:independance} addresses the propagation of independence. We introduce an intermediate particle system  and provide a detailed control of the approximation error. We also compare the different graphon-based reformulations arising in this context.
The mean-field limit is established in Section \ref{sec:mfl}.
Finally, Section \ref{sec:link} discusses the connections between the continuum (graph-limit) formulation and the mean-field limit.

\section{Presentation of the model and main results} \label{sec:model}
\subsection{Setting}

In the following, we indifferently denote $|x|$ the absolute value or the norm on $\R^d$ depending on $x$ being in $\R$ or $\R^d$. We denote $B(0,r)$ the open ball of radius $r$ and $I$ the interval $[0,1]$.

Throughout the paper, we shall assume the classical regularity conditions. From this point onward, the interaction function will be required to satisfy the following assumptions:
\begin{hyp}\label{hyp:phi}
The interaction function $ \phi: \R_+  \times  \R^d \times \R^d  \to \R $  is bounded continuous  and Lispchitz in $(x,y)$ uniformly in $t$ i.e. there exists $L_\phi >0$ such that, for all $x_1,x_2,y_1,y_2 \in \R$
\begin{equation}
\left|\phi(t,x_1,y_1) - \phi(t,x_2,y_2)\right| \leq L_\phi \left( |x_1- x_2| + |y_1 - y_2| \right)
\end{equation}  and 
\begin{equation} \label{eq:bound_phi}
\|\phi\|_\infty := M_\phi < + \infty.
\end{equation}
We also assume that $\phi(0)=0$. 
\end{hyp}

With regard to the edge dynamics, we make the following assumption:
\begin{hyp}\label{hyp:psi}
The function $\Lambda: I^2 \times  \R^d \times \R^d  \times \R  \times  \R^d \times \R^d  \times \R   \to \R $  is Lispchitz in $(x,y,w,\tx,\ty,\tx)$ uniformly in $(\xi,\zeta)$ i.e. there exists $L_\Lambda>0$ such that, for a. e. $\xi,\zeta \in I$, for all $(x_1,y_1,w_1,\tx_1,\ty_1,\tw_1), (x_2,y_2,w_2,\tx_2,\ty_2,\tw_2) \in   \R^d \times \R^d  \times \R  \times  \R^d \times \R^d  \times \R$, we have
\begin{multline}
 \left|\Lambda(\xi,\zeta,x_1,y_1,w_1,\tx_1,\ty_1,\tw_1) - \Lambda(\xi,\zeta,x_2,y_2,w_2,\tx_2,\ty_2,\tw_2)\right| \\  \leq L_\Lambda \left(  |x_1- x_2| + |y_1 - y_2| + |w_1 - w_2| + |\tx_1- \tx_2| + |\ty_1 - \ty_2| + |\tw_1 - \tw_2|  \right) 
\end{multline} 
We also assume the following bound on $\Lambda$:  there exists $C_\Lambda>0$ such that for all  
\begin{equation}\label{eq:psisublin}
|\Lambda(\xi,\zeta,x,y,w,\tx,\ty,\tw)| \leq C_\Lambda (1+|w|).
\end{equation}
\end{hyp}
\begin{rem} 
Concerning the boundedness assumption on the weight dynamics stated  in equation \eqref{eq:psisublin}, which was already present in \cite{AyiPouradierDuteil21,Throm_2024}, although it may appear restrictive at first glance, its  purpose is to ensure that no finite-time blow-up occurs. This condition is consistent with our intention to work within a framework analogous to that of \cite{Medvedev14}, where the weights are assumed to belong to the space  $L^\infty$.
\end{rem}
We impose the following structural condition on $\Lambda$.
\begin{hyp} \label{hyp:struct_lambda}
The function $\Lambda: I^2 \times   \R^d \times \R^d \times \R \times    \R^d \times \R^d \times \R   \to \R $ is equal to $0$ as soon as two of its  last six  variables are equal.
\end{hyp}

\begin{rem} 
This hypothesis rules out contributions for which two of the last six arguments coincide, i.e. it excludes interaction terms where the auxiliary pair $(i_1,j_1)$ shares an index either with itself or with $(i,j)$. From a modeling standpoint, this prevents overlapping or double-counted effects arising from identical or partially coinciding links, thereby ensuring that each connection evolves solely under the influence of distinct pairs of particles in the network. Although this assumption will be used throughout the paper, it only plays a technical role in Section \ref{sec:independance}, where it is crucial for establishing the independence results.
\end{rem}
\subsection{Functional spaces} \label{sec:spaces}
We denote  $\text{BL}(X)$ the space of bounded-Lipschitz functions  $\phi$ such that the bounded-Lipschitz norm is finite {\it i.e.}
    $$
\| \phi \|_{\text{BL}} := \max \{ \| \phi\|_{L^\infty}, [\phi]_{\rm Lip}\} < + \infty
$$
 where $ [\phi]_{\rm Lip}$ stands for the Lipschitz constant of $\phi$.  We denote $\mathcal{P}(X)$ the space of probability measure on $X$ and we equipped it with the following distance: for $\mu_1, \mu_2 \in \mathcal{P}(X)$, 
 $$d_{\rm BL}(\mu_1,\mu_2)
=\sup_{\Vert \varphi\Vert_{\rm BL}\leq 1}\int_{X}\varphi(x)(d\mu_1(x)-d\mu_2(x)).$$

The solution to the Vlasov-type equation will actually consist of parametrized families of probability densities. One of the novelty, which is due to our context, is that it will be parametrized by two parameters. 
\begin{definition}[Borel family of probability measures]
Consider any $\nu$ in $\mathcal{P}(I)$ and let $(\mu^{\xi,\zeta})_{\xi,\zeta \in I} \in \mathcal{P}(\mathbb{R}^d \times \mathbb{R}^d \times \mathbb{R})$ be a parametrized family of probability measures defined for $\nu$-a.e. $\xi,\zeta \in I$. We say that $(\mu^{\xi,\zeta})_{\xi,\zeta \in I} $ is a Borel family if the map $(\xi,\zeta) \in I^2 \mapsto \mu^{\xi,\zeta}(B)$ is Borel-measurable for every Borel set $B \subset \R^d \times \mathbb{R}^d \times \mathbb{R}$.
\end{definition}
We also introduce the class of fibered probability measures. 
\begin{definition}[Fibered probability measures]
Consider any $\nu$ in $\mathcal{P}(I)$ . We define the space of fibered probability measures $$\mathcal{P}_\nu( I^2 \times \mathbb{R}^d \times \mathbb{R}^d \times \mathbb{R}) : = \{ \mu \in \mathcal{P}( I^2 \times \mathbb{R}^d \times \mathbb{R}^d \times \mathbb{R} ) : \Pi_{\xi,\zeta} \# \mu = \nu \otimes \nu\}$$
where $\Pi_{\xi,\zeta} (\xi, \zeta,x,y,w) = (\xi, \zeta)$ is the projection on the two first components.  We recall that for $\mu_1, \mu_2$ measures, $\mathcal{T}$ a measurable map and $A$ a measurable set, $\mu_1= \mathcal{T}\# \mu_2$  means $$ \mu_1(A) = \mu_2(\mathcal{T}^{-1}(A)).$$ Thus, $\Pi_{\xi,\zeta}  \# \mu$  stands for the marginal of $\mu$ in the two first components.
\end{definition}
There is actually a link between these two spaces of measures thanks to the classical disintegration theorem. 
\begin{theorem}[Disintegration theorem] \label{theo:disintegration}
Consider any $\nu$ in $\mathcal{P}(I)$ and $\mu \in \mathcal{P}_\nu( I^2 \times \mathbb{R}^d \times \mathbb{R}^d \times \mathbb{R})$. Then, there exists an almost everywhere uniquely defined Borel family $(\mu^{\xi,\zeta})_{\xi,\zeta \in I} \subset \mathcal{P}(\mathbb{R}^d \times \mathbb{R}^d \times \mathbb{R})$ so that 
\begin{multline} \label{eq:disintegration}
\int_{ I^2 \times \mathbb{R}^d \times \mathbb{R}^d \times \mathbb{R}} \varphi(\xi,\zeta,x,y,w) \mu(d\xi,d\zeta,dx,dy,dw) \\
= \int_I \int_I  \left( \int_{\mathbb{R}^d \times \mathbb{R}^d \times \mathbb{R}}  \varphi(\xi,\zeta,x,y,w) \mu^{\xi,\zeta}(dx,dy,dw) \right)  \nu(d\xi)  \nu(d\zeta) 
\end{multline}
for every bounded Borel-measure map $\varphi : I^2 \times \R^d \times  \R^d \times \R \to \R$. 

Conversely, for any Borel family $(\mu^{\xi,\zeta})_{\xi,\zeta \in I} \subset \mathcal{P}(\mathbb{R}^d \times \mathbb{R}^d \times \mathbb{R})$, we can associate a unique fibered probabilty measure $\mu \in \mathcal{P}_\nu( I^2 \times \mathbb{R}^d \times \mathbb{R}^d \times \mathbb{R}) $ so that \eqref{eq:disintegration} holds true. This can be summarize by the writing $$\mu(\xi,\zeta,x,y,w)  = \mu^{\xi,\zeta}(x,y,w) \otimes \nu(\xi)  \otimes \nu(\zeta) $$
\end{theorem}
As a consequence, we will often identify measures $\mu \in \mathcal{P}_\nu( I^2 \times \mathbb{R}^d \times \mathbb{R}^d \times \mathbb{R})$ with their associated Borel families of almost everywhere defined measures $(\mu^{\xi,\zeta})_{\xi,\zeta \in I} \subset \mathcal{P}(\mathbb{R}^d \times \mathbb{R}^d \times \mathbb{R})$.

Similarly to what we have done in \cite{AyiPoyatoPouradierDuteil}, we choose to consider the following topology. 
\begin{definition}
Consider any $\nu \in \mathcal{P}(I)$, we define
$$\mathcal{P}_{1,\nu}( I^2 \times \mathbb{R}^d \times \mathbb{R}^d \times \mathbb{R}):=\left\{\mu\in \mathcal{P}_\nu( I^2 \times \mathbb{R}^d \times \mathbb{R}^d \times \mathbb{R} ):\, d_{1,\nu}(\mu,\delta_0)\,<\infty\right\}$$
with the $L^1$-bounded Lipschitz distance defined as 
$$d_{1,\nu}(\mu, \bar \mu) := \int_{I^2} d_{\rm BL}(\mu^{\xi,\zeta},\bar \mu^{\xi,\zeta})  \nu(d\xi)  \nu(d\zeta)$$
for any $\mu$, $\bar \mu \in \mathcal{P}_{1,\nu}(I^2 \times \mathbb{R}^d \times \mathbb{R}^d \times \mathbb{R} )$. 
\end{definition}
In all the rest of the article, we will take $\nu$ as the Lebesgue measure and for $\mu, \bar \mu \in \mathcal{P}_{1,\nu}( I^2 \times \mathbb{R}^d \times \mathbb{R}^d \times \mathbb{R})$, we will define the $L^1$-bounded Lipschitz distance as
$$d_{1}(\mu, \bar \mu) := \int_{I^2} d_{\rm BL}(\mu^{\xi,\zeta},\bar \mu^{\xi,\zeta}) d \xi d \zeta.$$
\begin{rem} 
It is natural and coherent to restrict ourselves to the Lebesgue measure on $I$ since the classical graph theory framework, which we aim to adopt, primarily focuses on graphs with weighted edges but unweighted vertices. In contrast, other works have investigated more general settings where heterogeneous weights are assigned to the vertices, which naturally leads to the choice of a non-uniform measure $\nu$ (see \cite{KuehnXu22} for instance). 
\end{rem}
We also introduce the space $\mathcal{C}([0,T],  \mathcal{P}_{1,\nu}(I^2 \times \mathbb{R}^d \times \mathbb{R}^d \times \mathbb{R}))$ endowed with the following norms, for $\mu, \bar \mu \in \mathcal{C}([0,T],  \mathcal{P}_{1,\nu}(I^2 \times \mathbb{R}^d \times \mathbb{R}^d \times \mathbb{R}))$, 
$$d_1^\infty(\mu, \bar \mu) := \sup_{t \in [0,T]} d_{1}(\mu_t, \bar \mu_t) $$
and the weighted version 
$$d_1^\alpha(\mu, \bar \mu) := \sup_{t \in [0,T]} e^{-\alpha t} d_{1}(\mu_t, \bar \mu_t) $$
for $\alpha >0$ to be chosen later. We start by noticing that, since $t \in [0,T]$, those two norms are equivalent. 

When working with a full curve $t \in [0,T] \mapsto \mu_t \in \mathcal{P}_{1,\nu}( I^2 \times \mathbb{R}^d \times \mathbb{R}^d \times \mathbb{R})$ and not an isolated measure $\mu \in \mathcal{P}_{1,\nu}( I^2 \times \mathbb{R}^d \times \mathbb{R}^d \times \mathbb{R})$, a time-dependent disintegration can also be obtained.
\begin{theorem}[Time-dependent disintegrations]\label{theo:disintegration-time-dependent}
Consider any $\nu\in \mathcal{P}(I)$ and any Borel family of probability measures $(\mu_t)_{t\in [0,T]}\subset \mathcal{P}_\nu( I^2 \times \mathbb{R}^d \times \mathbb{R}^d \times \mathbb{R} )$. Then, there exists a Borel family $(\mu_t^{\xi,\zeta})_{(t,\xi,\zeta)\in [0,T]\times I^2}\subset \mathcal{P}(\mathbb{R}^d \times \mathbb{R}^d \times \mathbb{R})$ defined for $dt\otimes\nu \otimes\nu$-a.e. $(t,\xi,\zeta)\in[0,T]\times I^2$ such that
\begin{multline}\label{eq:disintegration-time-dependent}
\int_0^T\left(\int_{\mathbb{R}^d\times I^2}\psi(t,\xi,\zeta,x,y,w)\,\mu_t(d\xi,d\zeta,dx,dy,dw)\right)\,dt \\ =\int_{[0,T]\times  I^2}\left(\int_{\mathbb{R}^d}\psi(t,\xi,\zeta,x,y,w)\,)\,\mu_t^{\xi,\zeta}(dx,dy,dw)\right)\,dt\,\nu(d\xi)\,\nu(d\zeta),
\end{multline}
for every bounded Borel-measurable map $\psi:[0,T]\times I^2 \times  \mathbb{R}^d \times  \mathbb{R}^d \times  \mathbb{R} \longrightarrow \mathbb{R}$. In particular,
\begin{enumerate} 
\item For a.e. $t\in [0,T]$, the slice $(\mu_t^{\xi,\zeta})_{\xi, \zeta\in I}$ is a Borel family defined for $\nu$-a.e. $\xi, \zeta \in I$, and it corresponds to a possible disintegration of $\mu_t$ in the sense of Theorem \ref{theo:disintegration}.
\item For $\nu$-a.e. $\xi, \zeta \in I$, the slice $(\mu_t^{\xi,\zeta})_{t\in [0,T]}$ is a Borel family defined for a.e. $t\in [0,T]$.
\end{enumerate}
\end{theorem}
\begin{proof}
A straightforward adaptation of the proof in \cite{AyiPoyatoPouradierDuteil} gives the result.
\end{proof}
We finish by stating some completeness properties.
\begin{prop} For any  $\nu \in \mathcal{P}(I)$, the spaces $$(\mathcal{P}_{1,\nu}( I^2 \times \mathbb{R}^d \times \mathbb{R}^d \times \mathbb{R}), d_1),$$  $$(\mathcal{C}([0,T],  \mathcal{P}_{1,\nu}(I^2 \times \mathbb{R}^d \times \mathbb{R}^d \times \mathbb{R})), d_1^\infty),$$  and $$(\mathcal{C}([0,T],  \mathcal{P}_{1,\nu}(I^2 \times \mathbb{R}^d \times \mathbb{R}^d \times \mathbb{R})), d_1^\alpha))$$ are complete.
\end{prop}

\begin{rem} 
To establish the completeness of 
$(\mathcal{P}_{1,\nu}(I^2 \times \mathbb{R}^d \times \mathbb{R}^d \times \mathbb{R}), d_1)$, 
we refer to \cite[Proposition A.3]{Peszek}, where the result is proved for 
$(\mathcal{P}_{2,\nu}(I^2 \times \mathbb{R}^d \times \mathbb{R}^d \times \mathbb{R}), d_2)$, 
i.e. in the case of the $L^2$ norm in $(\xi,\zeta)$ instead of $L^1$. 
A straightforward adaptation of the argument yields the desired result. 
It is then standard and easy to verify that this completeness property extends to 
$(\mathcal{C}([0,T], \mathcal{P}_{1,\nu}(I^2 \times \mathbb{R}^d \times \mathbb{R}^d \times \mathbb{R})), d_1^\infty)$. 
Finally, the completeness of 
$(\mathcal{C}([0,T], \mathcal{P}_{1,\nu}(I^2 \times \mathbb{R}^d \times \mathbb{R}^d \times \mathbb{R})), d_1^\alpha)$
follows immediately from the equivalence between the metrics $d_1^\infty$ and $d_1^\alpha$.
\end{rem}

\subsection{Statement of the results}

We now state the main results of this paper, which establish the derivation of a Vlasov-type equation in an extended phase space as the large-population limit of the interacting particle system under consideration.

An important aspect of our analysis is that this derivation can be achieved through two distinct and complementary approaches, reflecting different modeling and analytical frameworks. The first approach follows the classical mean-field methodology introduced by Sznitman. In this setting, we impose the additional assumption that the weight dynamics $\Lambda_{ij}$ is independent of the spatial variable $X_i$. As will be discussed throughout the paper, this restriction remains well justified from a modeling viewpoint and allows for a direct probabilistic treatment of the system.

The second approach relies on the graph limit perspective, in a deterministic setting. This approach makes it possible to remove the restrictive assumption on $\Lambda_{ij}$ and to handle more general interaction structures.

The first theorem below establishes the mean-field limit under the Sznitman-type framework, while the second theorem derives the corresponding limit equation using the graph-limit approach. Together, these results provide a comprehensive description of the asymptotic behavior of the system and clarify the connections between the two limiting procedures.

\begin{theorem}
 Let the interaction function $\phi$  satisfy Hypothesis \ref{hyp:phi}. Let the weight dynamics  $\Lambda$ satisfy Hypotheses \ref{hyp:psi} and \ref{hyp:struct_lambda}, and assume that it does not depend on the variable $x$.   Let the particle systems \eqref{eq:particle_syst_restrictive} be initialized with random variables
$(X^{N,0},W^{N,0})$ such that the family
$\{X_i^{N,0},\, W_{ij}^{N,0} : 1\le i,j\le N\}$ is independent,
$\E|X_i^{N,0}|<+\infty$, and $W^{N,0}$ satisfies
Hypothesis~\ref{hyp:graph_initial}. Consider the unique solution associated $(X^{N}, W^{N})$. Then,  the mean-field limit of the multi-agent system \eqref{eq:particle_syst_restrictive} is characterized in a suitable sense by a solution to the Vlasov-type equation \eqref{eq:vlasov-equation}-\eqref{eq:vlasov-force} for some  $(\mu_t^{\xi,\zeta})_{(t,\xi,\zeta)\in [0,T]\times I^2}\subset \mathcal{P}(\mathbb{R}^d \times \mathbb{R}^d \times \mathbb{R})$.
\end{theorem}
For a precise and fully rigorous statement, we refer to Theorem \ref{theo:main}.

\begin{theorem}
 Let the interaction function $\phi$  satisfy Hypothesis \ref{hyp:phi}. Let the weight dynamics $\Lambda$ satisfy Hypotheses \ref{hyp:psi} and \ref{hyp:struct_lambda}. Let $x_0 \in   L^\infty(I;\R^d)$, $ w_0 \in   L^\infty(I^2)$ and $(x,w)$ be the solution to the graph limit equation \eqref{eq:GL} with weight dynamics $\Lambda$ and initial conditions given by $x(0, \cdot) = x_0$ and $w(0, \cdot, \cdot) = w_0$. 
 Then,  the ``empirical measure'' $\mu^N$ associated  to the multi-agent system \eqref{eq:particle_syst_restrictive} converges weakly to  the ``continuous'' empirical measure $\tilde{\mu}$ associated to the continuum (or graph) limit of the multi-agent system \eqref{eq:particle_syst_restrictive} and  $\tilde{\mu}$ is a solution to the Vlasov-type  equation \eqref{eq:vlasov-equation}-\eqref{eq:vlasov-force}.
\end{theorem}
Similarly, we refer to Theorem~\ref{prop:convmu} for a complete and rigorous formulation.
\section{Well-posedness of the Vlasov-type equation} \label{sec:well-posedness}
\subsection{Well-posedness of the characteristics}

Our strategy for addressing the well-posedness of the Vlasov-type equation \eqref{eq:vlasov-equation}-\eqref{eq:vlasov-force} relies on the classical method of characteristics. This approach consists in determining the characteristic curves along which the original partial differential equation reduces to an ordinary differential equation. In the present framework, the characteristics associated with the Vlasov-type equation \eqref{eq:vlasov-equation}-\eqref{eq:vlasov-force} take the form
\begin{equation} \label{eq:charac}
\frac{d}{dt} Z [\mu](t,\xi,\zeta,x,y,w) = F_\Lambda[\mu_t](\xi,\zeta,Z [\mu](t,\xi,\zeta,x,y,w)) 
\end{equation} 
 where $$Z[\mu] := \left( \begin{array}{c}
X[\mu] \\ 
Y[\mu] \\ 
W[\mu]
 \end{array} \right) $$ and   $$Z[\mu](0,\xi,\zeta,x,y,w) = \left( \begin{array}{c}
x\\
y\\
w
 \end{array} \right). $$

\begin{prop}[Well-posed characteristics]\label{pro:well-posed-characteristics}
Let the interaction function $\phi$  satisfy Hypothesis \ref{hyp:phi}. Let the weight dynamics $\Lambda$ satisfy Hypotheses \ref{hyp:psi} and \ref{hyp:struct_lambda}. Consider any Borel family $(\mu_t)_{t \in [0,T]} \subset \mathcal{P}_{1,\nu}(I^2 \times \mathbb{R}^d \times \mathbb{R}^d \times \mathbb{R})$ such that 
 for all $t \in [0,T]$, we have  
\begin{equation}  \label{eq:moment_borne}
   \esssup_{\xi, \zeta \in I}  \int_{ \R^d \times \R^d \times \R} |\tw| \mu_t^{\xi,\zeta }(d\tx,d\ty,d\tw)   \leq \tilde C
\end{equation} 
 for some $\tilde C >0$.  
 
 Then, for all $t \in [0,T]$, there is a unique Carathéodory  solution $Z [\mu](t,\xi,\zeta,x,y,w)$ to the characteristic system \eqref{eq:charac}
for all $(x,y,w) \in \mathbb{R}^d \times \mathbb{R}^d \times \mathbb{R}$ and a.e. $\xi, \zeta \in I$.
\end{prop}
The proof relies on an application of the Carathéodory theorem. Let us first prove the Lipschitz property of the force. 
\begin{lemma} \label{lem:Lip_force2}
Suppose that the assumptions of Proposition \ref{pro:well-posed-characteristics} hold.  Then, for all $t \in [0,T]$, the force $F_\Lambda[\mu_t](\xi,\zeta,x,y,w)$ is Lipschitz with respect to $(x,y,w)$ i.e. for a.e. $\xi, \zeta \in I$, for $(x_1,y_1,w_1), (x_2,y_2,w_2)  \in \mathbb{R}^d \times \mathbb{R}^d \times \mathbb{R}$, we have 
\begin{equation}
\left| F_\Lambda[\mu_t](\xi,\zeta,x_1,y_1,w_1) - F_\Lambda[\mu_t](\xi,\zeta,x_2,y_2,w_2)\right| \leq \tilde C_1(T) \left(|x_1-x_2| + |y_1-y_2| + |w_1-w_2| \right)
\end{equation}
where $$\tilde C_1(T) := \max \left(L_\phi \tilde C,  L_{\Lambda}\right).$$ 
\end{lemma}

\begin{proof}
For a.e. $\xi, \zeta \in I$, for $(x_1,y_1,w_1), (x_2,y_2,w_2)  \in \mathbb{R}^d \times \mathbb{R}^d \times \mathbb{R}$, we have  
$$\begin{array}{l}
\dsp \left| F_\Lambda[\mu_t](\xi,\zeta,x_1,y_1,w_1) - F_\Lambda[\mu_t](\xi,\zeta,x_2,y_2,w_2)\right| \\
\dsp \leq  \left|  \intt \tw \phi(t,x_1,\ty)  \mu_t^{\xi,\tz} (d\tx,d\ty, d\tw) d\tz -  \intt \tw \phi(t,x_2,\ty)  \mu_t^{\xi,\tz} (d\tx,d\ty, d\tw) d\tz \right|\\
+\dsp  \left|  \intt \tw \phi(t,y_1,\ty) \mu_t^{\zeta,\tz} (d\tx,d\ty, d\tw) d\tz  -  \intt \tw \phi(t,y_2,\ty) \mu_t^{\zeta,\tz} (d\tx,d\ty, d\tw) d\tz \right|\\
+  \dsp  \left|   \inttt \Lambda( \xi, \zeta, x_1,y_1,w_1, \tx, \ty, \tw) \mu_t^\xztt(d\tx,d\ty,d\tw)d \txi d \tz \right. \\
  \dsp  \left. ~~~~~~~~~~~~~~~~~~~~~~~~~~~~~~~~~~~~~~~~~~-  \inttt \Lambda( \xi, \zeta, x_2,y_2,w_2, \tx, \ty, \tw) \mu_t^\xztt(d\tx,d\ty,d\tw)d \txi d \tz \right| \\
\dsp \leq   L_\phi \left( |x_1-x_2| + |y_1 - y_2| \right) \int_{I \times \R^d \times \R^d \times \R} |\tw| \mu_t^{\xi,\tz}(d\tx,d\ty,d\tw) d\tz   + L_\Lambda \left( |x_1-x_2| + |y_1 - y_2| + |w_1 - w_2| \right) 
\end{array}$$
where the last inequality is obtained using the regularity assumptions of the functions $\phi$ and $\Lambda$ in Hypothesis \ref{hyp:phi} and  \ref{hyp:psi}. We conclude by using  the bound on equation \eqref{eq:moment_borne}. 
\end{proof}

\begin{proof}[Proof of Proposition~\ref{pro:well-posed-characteristics}]
By a similar computation as in Lemma~\ref{lem:Lip_force2}, we obtain the linear growth estimate
\[
\left| F_\Lambda[\mu_t](\xi,\zeta,x,y,w) \right|
\le \tilde C_1(T)\bigl(1+|x|+|y|+|w|\bigr).
\]
Moreover, Lemma~\ref{lem:Lip_force2} ensures that the map
$(x,y,w) \mapsto F_\Lambda[\mu_t](\xi,\zeta,x,y,w)$ is Lipschitz uniformly in $t\in[0,T]$.
Finally, for a.e. $(\xi,\zeta)\in I^2$ and every $(x,y,w)\in\R^d\times\R^d\times\R$,
the map
\[
t\longmapsto F_\Lambda[\mu_t](\xi,\zeta,x,y,w)
\]
is Borel measurable, since it is defined as an integral of a Borel integrand
against the Borel family $(\mu_t)_{t\in[0,T]}$.
Therefore, $F_\Lambda[\mu_t](\xi,\zeta,\cdot)$ is a Carathéodory vector field.

Applying the Carathéodory existence and uniqueness theorem, we deduce that,
for a.e. $(\xi,\zeta)\in I^2$ and every $(x,y,w)\in\R^d\times\R^d\times\R$,
the characteristic system~\eqref{eq:charac} admits a unique absolutely continuous solution
$t\mapsto Z[\mu](t,\xi,\zeta,x,y,w)$ on $[0,T]$, which satisfies the differential equation
for almost every $t\in[0,T]$.
\end{proof}

 In the rest of the paper, we will use the following notation:
\begin{equation}\label{eq:characteristics-flow-map}
\mathcal{T}_t^{\xi,\zeta}[\mu](x,y,w):=Z [\mu](t,\xi,\zeta,x,y,w), \quad t\in [0,T],\,(x,y,w) \in \mathbb{R}^d \times \mathbb{R}^d \times \mathbb{R},\,\mbox{a.e. }\xi, \zeta\in I.
\end{equation}
{We actually have some a priori estimates.}
\begin{lemma} \label{lem:aprioriestimates}
Suppose that $\mu\in C([0,T],\mathcal{P}_{1,\nu}(I^2 \times \R^d \times \R^d \times \R))$ is a solution  to \eqref{eq:vlasov-equation}-\eqref{eq:vlasov-force} issued at $\mu_0$ which satisfies 
 $$ \esssup_{\xi, \zeta \in I} \int_{ \R^d \times \R^d \times \R}   w^2 \mu_0^{\xi,\zeta}(dx,dy,dw)  < + \infty.$$

 Then, for all $t \in [0,T]$, we have $$  \esssup_{\xi, \zeta \in I} \int_{ \R^d \times \R^d \times \R}  w^2 \mu_t^{\xi,\zeta}(dx,dy,dw)  < + \infty.$$
\end{lemma}
\begin{proof}
We multiply equation \eqref{eq:vlasov-equation} by $w^2$ and we integrate in $(x,y,w)$ to obtain
$$\begin{array}{rcl}
\dsp \frac{d}{dt}   \int_{ \R^d \times \R^d \times \R}  w^2 \mu_t^{\xi,\zeta}(dx,dy,dw)  & = &    \dsp 2  \int_{ \R^d \times \R^d \times \R}  w \left(\inttt \Lambda( \xi, \zeta, x,y,w, \tx, \ty, \tw) \mu_t^\xztt(d\tx,d\ty,d\tw)d \txi d \tz\right)\\  & &   ~~~~~~~~~~~~~~~~~~~~~~~~~~~~~~~~~~~~~~~~~~~~~~~~~~~~~~~~~~~~~~~~~~~~~~\mu_t^{\xi,\zeta}(dx,dy,dw)  \\
& \leq & \dsp  2 {C_\Lambda} \int_{ \R^d \times \R^d \times \R}   (|w| + w^2)  \mu_t^{\xi,\zeta}(dx,dy,dw)  \\
& \leq & \dsp   4{C_\Lambda} \left( 1 + \int_{ \R^d \times \R^d \times \R}   w^2  \mu_t^{\xi,\zeta}(dx,dy,dw) \right)
\end{array}$$
where we have used $$\int_{ \R^d \times \R^d \times \R}   |w|  \mu_t^{\xi,\zeta}(dx,dy,dw) \leq 1 + \int_{ \R^d \times \R^d \times \R}   w^2  \mu_t^{\xi,\zeta}(dx,dy,dw)$$ to obtain the last inequality. Thus, by Gronwall's lemma, we deduce that, for all $t \in [0,T]$, $$\int_{ \R^d \times \R^d \times \R}  w^2 \mu_t^{\xi,\zeta}(dx,dy,dw)  \leq  \left(\int_{ \R^d \times \R^d \times \R}   w^2  \mu_0^{\xi,\zeta}(dx,dy,dw) +  4{C_\Lambda}T \right) \exp (4{C_\Lambda}T)$$
which concludes the proof by first bounding the right-hand side by the essential supremum in $\xi, \zeta$ and then taking this same essential supremum on the left-hand side. 
\end{proof}
Thus, we can actually rewrite the Lipschitz property of the force $F_\Lambda[\mu_t]$ in the case where  $\mu\in C([0,T],\mathcal{P}_{1,\nu}(I^2 \times \R^d \times \R^d \times \R))$ is a solution  to \eqref{eq:vlasov-equation}-\eqref{eq:vlasov-force}.
\begin{lemma} \label{lem:Lip_force}
Let the interaction function $\phi$  satisfy Hypothesis \ref{hyp:phi}. Let the weight dynamics $\Lambda$ satisfy Hypotheses \ref{hyp:psi} and \ref{hyp:struct_lambda}. Suppose that  $\mu\in C([0,T],\mathcal{P}_{1,\nu}(I^2 \times \R^d \times \R^d \times \R))$ is a solution  to \eqref{eq:vlasov-equation}-\eqref{eq:vlasov-force} issued at $\mu_0$ which satisfies 
 $$    \esssup_{\xi, \zeta \in I}   \int_{ \R^d \times \R^d \times \R}   w^2 \mu_0^{\xi,\zeta}(dx,dy,dw)  < + \infty.$$
 
 Then, for all $t \in [0,T]$, the force $F_\Lambda[\mu_t](\xi,\zeta,x,y,w)$ is Lipschitz with respect to $(x,y,w)$ i.e. for a.e. $\xi, \zeta \in I$, for $(x_1,y_1,w_1), (x_2,y_2,w_2)  \in \mathbb{R}^d \times \mathbb{R}^d \times \mathbb{R}$, we have 
\begin{equation}
\left| F_\Lambda[\mu_t](\xi,\zeta,x_1,y_1,w_1) - F_\Lambda[\mu_t](\xi,\zeta,x_2,y_2,w_2)\right| \leq  C_1(T) \left(|x_1-x_2| + |y_1-y_2| + |w_1-w_2| \right)
\end{equation}
where $$C_1(T) := \max \left(L_\phi  \left( 2 + \left(  \esssup_{\xi, \zeta \in I}    \int_{ \R^d \times \R^d \times \R}   w^2 \mu_0^{\xi,\zeta}(dx,dy,dw) +  4{C_\Lambda}T \right) \exp (4{C_\Lambda}T) \right), L_{\Lambda}\right)$$ 
\end{lemma}

\begin{proof}
We follow the estimates of Lemma~\ref{lem:Lip_force2}. For a.e. $(\xi,\zeta)\in I^2$ and any
$(x_1,y_1,w_1),(x_2,y_2,w_2)\in\R^d\times\R^d\times\R$, we have
\begin{multline*}
\bigl|F_\Lambda[\mu_t](\xi,\zeta,x_1,y_1,w_1)-F_\Lambda[\mu_t](\xi,\zeta,x_2,y_2,w_2)\bigr|
\\  \leq  L_\phi \left( |x_1-x_2| + |y_1 - y_2| \right) \int_{I \times \R^d \times \R^d \times \R} |\tw| \mu_t^{\xi,\tz}(d\tx,d\ty,d\tw) d\tz   + L_\Lambda \left( |x_1-x_2| + |y_1 - y_2| + |w_1 - w_2| \right).
\end{multline*}

Since
\[
\int_{I\times\R^d\times\R^d\times\R}|\tilde w|\,\mu_t^{\xi,\tilde\zeta}(d\tilde x,d\tilde y,d\tilde w)\,d\tilde\zeta
\le \esssup_{\xi,\zeta\in I}\int |w|\,\mu_t^{\xi,\zeta}(dx,dy,dw),
\]
the Lipschitz constant $$C_1(T) = \max \left(L_\phi  \left( 2 + \left(  \esssup_{\xi, \zeta \in I}    \int_{ \R^d \times \R^d \times \R}   w^2 \mu_0^{\xi,\zeta}(dx,dy,dw) +  4{C_\Lambda}T \right) \exp (4{C_\Lambda}T) \right), L_{\Lambda}\right) $$ is obtained using the  elementary  inequality $$\int_{ \R^d \times \R^d \times \R}   |w|  \mu_t^{\xi,\zeta}(dx,dy,dw) \leq 1 + \int_{ \R^d \times \R^d \times \R}   w^2  \mu_t^{\xi,\zeta}(dx,dy,dw)$$ together with the bound provided by Lemma~\ref{lem:aprioriestimates}.
\end{proof}

\subsection{Well posedness of the Vlasov-type equation}

In this section, we will now specify the type of solutions we shall consider. To this end, let us introduce the notion of distributional solution. First, let us reformulate \eqref{eq:vlasov-equation}-\eqref{eq:vlasov-force}  in a form compatible with the notion of fibered probability measure $\mu \in \mathcal{P}_{\nu}( I^2 \times \mathbb{R}^d \times \mathbb{R}^d \times \mathbb{R})$ introduced in section \ref{sec:spaces}. It becomes
\begin{align}\label{eq:vlasov-equation2}
\left\{ \begin{aligned}
&\partial_t \mu_t +\divop_{x,y,w}(F_\Lambda[\mu_t]\mu_t)=0,\\
&\mu_{t=0}=\mu_0
\end{aligned} \right.
\end{align}
with $\mu_0 \in \mathcal{P}_{\nu}( I^2 \times \mathbb{R}^d \times \mathbb{R}^d \times \mathbb{R})$ and where $F_\Lambda[\mu_t]$ is defined in \eqref{eq:vlasov-force}.
\begin{definition}[Distributional solutions of \eqref{eq:vlasov-equation2}-\eqref{eq:vlasov-force}]\label{defi:distributional-solution-Vlasov-joint}
Consider any curve of probability measures $\mu\in (\mathcal{C}([0,T],  \mathcal{P}_{1,\nu}(I^2 \times \mathbb{R}^d \times \mathbb{R}^d \times \mathbb{R}))$ for some $T>0$. We say that $\mu$ is a distributional solution of the system \eqref{eq:vlasov-equation2}-\eqref{eq:vlasov-force} if
\begin{align}
&\int_0^T\int_{I^2 \times \mathbb{R}^d\times \mathbb{R}^d \times \mathbb{R}}\left(\partial_t\psi(t,\xi,\zeta,x,y,w)+F_\Lambda[\mu_t](\xi,\zeta,x,y,w)\cdot\nabla_{x,y,w} \psi(t,\xi,\zeta,x,y,w)\right)\\
&~~~~~~~~~~~~~~~~~~~~~~~~~~~~~~~~~~~~~~~~~~~~~~~~~~~~~~~~~~~~~~~~\,\mu_t(d\xi,d\zeta,dx,dy,dw)\,dt\nonumber \\
&\qquad=-\int_{I^2 \times \mathbb{R}^d\times \mathbb{R}^d \times \mathbb{R}} \psi(0,\xi,\zeta,x,y,w)\,\mu_0(d\xi,d\zeta,dx,dy,dw),\label{eq:distributional-solution-joint}
\end{align}
for all $\psi\in C^1_c([0,T)\times I^2 \times \mathbb{R}^d\times \mathbb{R}^d \times \mathbb{R})$.
\end{definition}
Standard arguments show that the solution to the Vlasov-type equation \eqref{eq:vlasov-equation2}-\eqref{eq:vlasov-force} can be written using the solution to the characteristics \eqref{eq:charac}  as 
\begin{equation} \label{eq:fixed_point}
\mu_t^{\xi,\zeta}= \mathcal{T}_t^{\xi,\zeta}[\mu] \# \mu_0 =: \mathcal{F}[\mu]_t^{\xi,\zeta},
\end{equation}
see for instance \cite[Proposition 4.5]{AyiPoyatoPouradierDuteil} in a non-adaptive hypergraph setting.
% where we recall that for $\mu_1, \mu_2$ measures, $T$ a measurable map and $A$ a measurable set, $\mu_1= T \# \mu_2$  means $$ \mu_1(A) = \mu_2(T^{-1}(A)).$$
Let us give a more precise statement.
\begin{prop}[Distributional solutions of \eqref{eq:vlasov-equation2}-\eqref{eq:vlasov-force}]\label{pro:distributional-solution-Vlasov-fibered}
Consider any curve of probability measures $\mu\in (\mathcal{C}([0,T],  \mathcal{P}_{1,\nu}(I^2 \times \mathbb{R}^d \times \mathbb{R}^d \times \mathbb{R}))$, and its associated Borel family $(\mu_t^{\xi,\zeta})_{(t,\xi,\zeta)\in [0,T]\times I^2}$ as in the disintegration Theorem \ref{theo:disintegration-time-dependent}. Then, the following conditions are equivalent:
\begin{enumerate}
\item $\mu$ is a distributional solution to \eqref{eq:vlasov-equation2}-\eqref{eq:vlasov-force}.
\item For a.e. $\xi,\zeta \in I$, $(\mu_t^{\xz})_{t\in [0,T]}$ is a distributional solution to \eqref{eq:vlasov-equation}-\eqref{eq:vlasov-force}, {\it i.e.},
\begin{multline}\label{eq:distributional-solution-fibered}
\int_0^T\int_{\mathbb{R}^d \times \mathbb{R}^d \times \mathbb{R}}\left(\partial_t\varphi(t,x,y,w)+F_\Lambda[\mu_t](t,\xi,\zeta,x,y,w)\cdot\nabla_{x,y,w} \varphi(t,x,y,w)\right)\,\mu_t^{\xi,\zeta}(dx,dy,dw)\,dt \\ =-\int_{\mathbb{R}^d \times \mathbb{R}^d \times \mathbb{R}}\varphi(0,x,y,w)\,\mu_0^{\xi,\zeta}(dx,dy,dw),
\end{multline}
for all $\varphi\in C^1_c([0,T)\times \mathbb{R}^d  \times \mathbb{R}^d \times \mathbb{R})$.
\item For a.e. $\xi, \zeta \in I$, $(\mu_t^\xz)_{t\in [0,T]}$ is the push forward along the flow map, {\it i.e.}, it satifies \eqref{eq:fixed_point}  $\mbox{a.e. }t\in [0,T]$
where $\mathcal{T}_t^\xi$ is given in \eqref{eq:characteristics-flow-map}.
\end{enumerate}
\end{prop}
Consequently, the proof of well-posedness of the Vlasov-type equation \eqref{eq:vlasov-equation2}-\eqref{eq:vlasov-force} is reduced to the resolution of the fixed-point problem \eqref{eq:fixed_point}.
In order to solve it, we shall crucially exploit the fact that the initial measure has compact support and that we are working on a finite time interval, which will allow us to obtain uniform bounds. Accordingly, we impose the following conditions on the initial measure:

\begin{hyp} \label{hyp:supp_compact}
The initial measure $\mu_0\in \mathcal{P}_{1,\nu}(I^2 \times \R^d \times \R^d \times \R)$  satisfies for almost every $\xi,\zeta$, $\text{supp}(\mu_0^{\xi,\zeta}) \subset B(0,R_X) \times B(0,R_Y) \times [-R_M,R_M]  $ for some constants $R_X,R_Y,R_M>0$.
\end{hyp}

\begin{rem}
Although restrictive from a mathematical standpoint, the assumption of bounded domains for both opinions and interaction weights is well motivated by empirical and modeling considerations. In survey-based studies, opinions are almost invariably collected on finite scales (Likert-type, 0–10 self-placement, etc.), see \cite{Gestefeld} for instance, or large-scale datasets such as European Social Survey.  Regarding interaction weights, empirical studies of social networks show that edge weights are constructed from normalized frequencies, probabilities, or intensities of interactions  (see for instance \cite{Hart,Mattie}), and thus naturally belong to bounded intervals. This provides strong justification for restricting our analysis to compact domains in both variables.
\end{rem}

 This context leads to the following result.
\begin{theorem}[Well-posedness of the Vlasov-type equation]\label{theo:well-posedness-vlasov}
 Let the interaction function $\phi$  satisfy Hypothesis \ref{hyp:phi}. Let the weight dynamics $\Lambda$ satisfy Hypotheses \ref{hyp:psi} and \ref{hyp:struct_lambda}. Let the initial measure $\mu_0$ satisfy Hypothesis \ref{hyp:supp_compact}. Then,  there exists a unique $\mu\in C([0,T],\mathcal{P}_{1,\nu}(I^2 \times \R^d \times \R^d \times \R))$ solution  to \eqref{eq:vlasov-equation}-\eqref{eq:vlasov-force} issued at $\mu_0$.
\end{theorem}

In the following proposition, we first observe that, since the evolution is considered over a finite time interval and the initial measure has compact support, the support of the solution necessarily remains bounded for all $t \in [0,T]$.
\begin{prop} \label{prop:compacite}  Let the interaction function $\phi$  satisfy Hypothesis \ref{hyp:phi}. Let the weight dynamics $\Lambda$ satisfy Hypotheses \ref{hyp:psi} and \ref{hyp:struct_lambda}. Let the initial measure $\mu_0$ satisfy Hypothesis \ref{hyp:supp_compact}.  Then, for all $t \in [0,T]$, for $\mu$ satisfying \eqref{eq:fixed_point},  we have $\text{supp}(\mu_t^{\xi,\zeta})  \subset B_T^{\phi, \Lambda}(R_X,R_Y,R_M)$ with 
\begin{multline*}
B_T^{\phi, \Lambda}(R_X,R_Y,R_M) := B(0,R_X + T M_\phi m_1) \times B(0,R_Y + T M_\phi m_1) \times [- \bar R_M , \bar R_M ]
\end{multline*}
with $\dsp m_1 := \left(2 + \left(   \esssup_{\xi, \zeta \in I}   \int_{ \R^d \times \R^d \times \R}   w^2 \mu_0^{\xi,\zeta}(dx,dy,dw) +  4{C_\Lambda}T \right) \exp (4{C_\Lambda}T) \right)$ and $\bar R_M :=  R_M (1+T C_\Lambda) 
 + T C_\Lambda$.
\end{prop}

\begin{proof}
Let $(x,y,w)\in \supp(\mu_0^{\xi,\zeta})$. By Hypothesis~\ref{hyp:supp_compact},
$|x|\le R_X$, $|y|\le R_Y$ and $|w|\le R_M$.
By definition of the characteristics, for all $t\in[0,T]$,

$\begin{aligned}
X[\mu](t,\xi,\zeta,x,y,w)  &= x + \int_0^t \intt \tilde w\,\phi(s,X[\mu](s,\xi,\zeta,x,y,w) ,\tilde y)\,
\mu_s^{\xi,\tilde\zeta}(d\tilde x,d\tilde y,d\tilde w)\,d\tilde\zeta\,ds,\\
Y[\mu](t,\xi,\zeta,x,y,w) &= y + \int_0^t \intt  \tilde w\,\phi(s,Y[\mu](s,\xi,\zeta,x,y,w),\tilde y)\,
\mu_s^{\zeta,\tilde\zeta}(d\tilde x,d\tilde y,d\tilde w)\,d\tilde\zeta\,ds,
\end{aligned}
$

\begin{multline}
W[\mu](t,\xi,\zeta,x,y,w)  \\
= w +
\int_0^t  \inttt 
\Lambda(\xi,\zeta,X[\mu](s,\xi,\zeta,x,y,w),Y[\mu](s,\xi,\zeta,x,y,w),W[\mu](s,\xi,\zeta,x,y,w)  ,\tilde x,\tilde y,\tilde w)\, \\
\mu_s^{\tilde\xi,\tilde\zeta}(d\tilde x,d\tilde y,d\tilde w)\,d\tilde\xi d\tilde\zeta\,ds.
\end{multline}

Using the bounds \eqref{eq:bound_phi} and \eqref{eq:psisublin} together with
Lemma~\ref{lem:aprioriestimates}, we obtain
\[
|X[\mu](t,\xi,\zeta,x,y,w)|\le R_X + T M_\phi m_1,\qquad
|Y[\mu](t,\xi,\zeta,x,y,w)|\le R_Y + T M_\phi m_1,
\]
and
\[
|W[\mu](t,\xi,\zeta,x,y,w) |\le R_M + T C_\Lambda(1+R_M)
= R_M(1+TC_\Lambda)+TC_\Lambda.
\]
Therefore,
\[
\supp(\mu_t^{\xi,\zeta}) \subset
B(0,R_X+TM_\phi m_1)\times B(0,R_Y+TM_\phi m_1)\times[-\bar R_M,\bar R_M],
\]
with $\bar R_M:=R_M(1+TC_\Lambda)+TC_\Lambda$, which concludes the proof.
\end{proof}

 Let us define $\tilde B_T^{\phi,\Lambda}(R_X,R_Y,R_M) := B(0,R_X + T M_\phi m_1+1) \times B(0,R_Y + T M_\phi m_1+1 ) \times [- (\bar R_M+1) , \bar R_M +1 ]$. We know that we can build a smooth function $\chi_T$ such that $$ 
 \chi_T : (x,y,w) \mapsto \left\{\begin{array}{l} 1 \text{ if } (x,y,w) \in  B_T^{\phi,\Lambda}(R_X,R_Y,R_M)\\
 0 \text{ if } (x,y,w) \notin \tilde B_T^{\phi,\Lambda}(R_X,R_Y,R_M)
 \end{array}\right. $$ such that $\chi_T \in \mathcal{C}^\infty(\R^d \times \R^d \times \R)$ (see \cite[Proposition 2.25]{Lee} for instance). The next lemma provides crucial bounds that will be instrumental in the proof of Theorem \ref{theo:well-posedness-vlasov}.
\begin{lemma} \label{lem:normes_BL}
For $(\xi, \zeta) \in I^2$,  $(x,y,w) \in \tilde B_T^{\phi,\Lambda}(R_X,R_Y,R_M) $, we define $$\varphi_1 :(\tx,\ty, \tw) \mapsto \tw \chi_T(\tx,\ty, \tw)$$ $$\varphi_2 :(\tx,\ty, \tw) \mapsto \chi_T(\tx,\ty, \tw) \Lambda(\xi,\zeta,x,y,w,\tx,\ty,\tw),$$  $$\varphi_3 :(\tx,\ty, \tw) \mapsto \varphi_1(\tx,\ty, \tw)  \phi(t,x,\ty)$$ and  $$\varphi_4 :(\tx,\ty, \tw) \mapsto \varphi_1(\tx,\ty, \tw)  \phi(t,y,\ty).$$ Then, we have $$\| \varphi_2 \|_{\text{BL}} \leq  L_2,  \, \| \varphi_3 \|_{\text{BL}} \leq  L_3 \text{ and }  
\| \varphi_4 \|_{\text{BL}} \leq L_3 $$ 
with $$L_2: =\max (C_\Lambda(2+\bar R_M) ,   \|\nabla \chi_T \|_{L^\infty} C_\Lambda(2+\bar R_M), L_\Lambda)$$ and $$\dsp L_3 := \max  \left( (\bar R_M + 1)M_\phi, M_\phi (1 + (\bar R_M +1)  \|\nabla \chi_T \|_{L^\infty}), (1+\bar R_M) L_\phi\right).$$
\end{lemma}
\begin{proof}
For any $(\tx,\ty,\tw), (\bar x, \bar y, \bar z) \in \R^d \times \R^d \times \R$, we have 
$$\begin{array}{l}
|\varphi_2 (\tx,\ty, \tw) - \varphi_2 (\bar x, \bar y, \bar z)| \\
\dsp \leq   |\chi_T(\tx,\ty, \tw) - \chi_T  (\bar x, \bar y, \bar z) |  |\Lambda(\xi,\zeta,x,y,w,\tx,\ty,\tw)| \\
~~~~~~~~~~~~~~~~~~~~~~~~~~~~~~~~~~+ |\chi_T  (\bar x, \bar y, \bar z) | |  \Lambda(\xi,\zeta,x,y,w,\tx,\ty,\tw) - \Lambda(\xi,\zeta,x,y,w, \bar x, \bar y, \bar z) |\\
\leq  \|\nabla \chi_T \|_{L^\infty} |(\tx,\ty,\tw) - (\bar x, \bar y, \bar z) | C_\Lambda(1+|w|) + L_\Lambda |(\tx,\ty,\tw) - (\bar x, \bar y, \bar z) | \\
\leq \max (  \|\nabla \chi_T \|_{L^\infty} C_\Lambda(2+\bar R_M), L_\Lambda) |(\tx,\ty,\tw) - (\bar x, \bar y, \bar z) |.
\end{array}$$

Moreover,
\[
\|\varphi_2\|_{L^\infty}
\le \|\chi_T\|_{L^\infty}  C_\Lambda(1+1+\bar R_M)
\le C_\Lambda(2+\bar R_M).
\]
Therefore,
\[
\|\varphi_2\|_{\mathrm{BL}}
\le \max\Bigl(C_\Lambda(2+\bar R_M),\,
\|\nabla\chi_T\|_{L^\infty}C_\Lambda(2+\bar R_M),\,L_\Lambda\Bigr)
=: L_2.
\]
Next, we observe that $$\|\nabla \varphi_1 \|_{L^\infty} \leq 1 + (\bar R_M +1)  \|\nabla \chi_T \|_{L^\infty}$$ and we get
$$\begin{array}{l}
|\varphi_3 (\tx,\ty, \tw) - \varphi_3 (\bar x, \bar y, \bar z)| \\
\leq |\varphi_1 (\tx,\ty, \tw) - \varphi_1 (\bar x, \bar y, \bar z)| M_\phi + \| \varphi_1\|_{L^\infty} L_\phi |\ty - \bar y| \\
\leq M_\phi (1 + (\bar R_M +1)  \|\nabla \chi_T \|_{L^\infty})  |(\tx,\ty,\tw) - (\bar x, \bar y, \bar z) |  +  \| \varphi_1\|_{L^\infty} L_\phi |\ty - \bar y| \\
\leq \max \left( M_\phi (1 + (\bar R_M +1)  \|\nabla \chi_T \|_{L^\infty}), (1+\bar R_M) L_\phi\right) |(\tx,\ty,\tw) - (\bar x, \bar y, \bar z) |.
\end{array}$$
Since $
\|\varphi_1\|_{L^\infty}
=\sup|\tw\,\chi_T(\tx,\ty,\tw)|
\le \bar R_M+1$, it follows that \[
\|\varphi_3\|_{L^\infty}
\le (\bar R_M+1)M_\phi.
\]
Hence,
\[
\|\varphi_3\|_{\mathrm{BL}}
\le \max\Bigl((\bar R_M+1)M_\phi,\,
M_\phi(1+(\bar R_M+1)\|\nabla\chi_T\|_{L^\infty}),\,
(\bar R_M+1)L_\phi\Bigr)
=: L_3.
\]
The same estimate holds for $\varphi_4$ by the same arguments, which concludes the proof.
\end{proof}

As previously mentioned, the proof of Theorem \ref{theo:well-posedness-vlasov} reduces to solving the fixed-point problem \eqref{eq:fixed_point}.
\begin{proof}[Proof of Theorem \ref{theo:well-posedness-vlasov}]
For any initial measure $\mu_*\in \mathcal{P}_{1,\nu}(I^2 \times \R^d \times \R^d \times \R)$  satisfying for almost every $\xi,\zeta$, $\text{supp}(\mu_*^{\xi,\zeta}) \subset B(0,R_X) \times B(0,R_Y) \times [-R_M,R_M]$, and any curve $\mu \in \mathcal{C}([0,T],  \mathcal{P}_{1,\nu}(I^2 \times \mathbb{R}^d \times \mathbb{R}^d \times \mathbb{R}))$, we define $\mathcal{F}[\mu] \in \mathcal{C}([0,T],  \mathcal{P}_{1,\nu}(I^2 \times \mathbb{R}^d \times \mathbb{R}^d \times \mathbb{R}))$ by 
$$ \mathcal{F}[\mu]_t^{\xi,\zeta} := \mathcal{T}_t^{\xi,\zeta}[\mu] \# \mu_*$$

Thus, for $\mu, \bar \mu  \in \mathcal{C}([0,T],  \mathcal{P}_{1,\nu}(I^2 \times \mathbb{R}^d \times \mathbb{R}^d \times \mathbb{R}))$, we have
$$ \begin{array}{l}
d_{\rm BL}(\mathcal{F}[\mu]_t^{\xi,\zeta},\mathcal{F}[\bar \mu]_t^{\xi,\zeta}) \\
\dsp \leq \sup_{\Vert \varphi\Vert_{\rm BL}\leq 1}\int_{\R^d \times \R^d \times \R}\varphi(x,y,w)\left(d\mathcal{F}[\mu]_t^{\xi,\zeta}(x,y,w)-d\mathcal{F}[\bar \mu]_t^{\xi,\zeta}) (x,y,w)\right)\\
\dsp = \sup_{\Vert \varphi\Vert_{\rm BL}\leq 1}\int_{\R^d \times \R^d \times \R}\varphi(\mathcal{T}^{\xi,\zeta}_t[\mu](x,y,w)) - \varphi(\mathcal{T}^{\xi,\zeta}_t[\bar \mu](x,y,w)) d\mu_*^{\xi,\zeta}(x,y,w)\\
\dsp \leq \sup_{\Vert \varphi\Vert_{\rm BL}\leq 1} \int_{\R^d \times \R^d \times \R} \left| Z[\mu](t,\xi,\zeta,x,y,w) - Z[\bar \mu](t,\xi,\zeta,x,y,w)\right|  d\mu_*^{\xi,\zeta}(x,y,w) 
\end{array}$$
Thus, we have
$$\begin{array}{l}
 \left| Z[\mu](t,\xi,\zeta,x,y,w) - Z[\bar \mu](t,\xi,\zeta,x,y,w)\right|  \\
\dsp = \left|  \int_0^t \left( F_\Lambda[\mu_s](\xi,\zeta,Z [\mu](s,\xi,\zeta,x,y,w))  - F_\Lambda[\bar \mu_s](\xi,\zeta,Z [\bar \mu](s,\xi,\zeta,x,y,w))  \right) ds  \right|\\
\dsp \leq \int_0^t \left| F_\Lambda[\mu_s](\xi,\zeta,Z [\mu](s,\xi,\zeta,x,y,w))  - F_\Lambda[ \mu_s](\xi,\zeta,Z [\bar \mu](s,\xi,\zeta,x,y,w)) \right| ds \\
+ \dsp \int_0^t \left| F_\Lambda[\mu_s](\xi,\zeta,Z [\bar \mu](s,\xi,\zeta,x,y,w))  - F_\Lambda[ \bar \mu_s](\xi,\zeta,Z [\bar \mu](s,\xi,\zeta,x,y,w)) \right| ds.
\end{array}$$
For the first term in the last inequality, we use the Lipschitz continuity of $F_\Lambda$ with respect to $(x,y,w)$ and get 
\begin{multline}
\int_0^t \left| F_\Lambda[\mu_s](\xi,\zeta,Z [\mu](s,\xi,\zeta,x,y,w))  - F_\Lambda[ \mu_s](\xi,\zeta,Z [\bar \mu](s,\xi,\zeta,x,y,w)) \right| ds  \\ \leq \int_0^t C_1(T)  \left| Z[\mu](s,\xi,\zeta,x,y,w) - Z[\bar \mu](s,\xi,\zeta,x,y,w)\right| ds.
\end{multline}
Regarding the second term, we get 
$$\begin{array}{l} 
 \dsp \int_0^t \left| F_\Lambda[\mu_s](\xi,\zeta,Z [\bar \mu](s,\xi,\zeta,x,y,w))  - F_\Lambda[ \bar \mu_s](\xi,\zeta,Z [\bar \mu](s,\xi,\zeta,x,y,w)) \right| ds \\
 = \displaystyle \int_0^t \Bigg( \left| \intt \tw \phi(s,X [\bar \mu](s,\xi,\zeta,x,y,w),\ty)  \left(\mu_s^{\xi,\tz} (d\tx,d\ty, d\tw) - \bar \mu_s^{\xi,\tz} (d\tx,d\ty, d\tw) \right) d\tz\right|  \\ 
 ~~~~~~~~~~ + \displaystyle \left|  \intt \tw \phi(s,Y [\bar \mu](s,\xi,\zeta,x,y,w),\ty) \left( \mu_s^{\zeta,\tz} (d\tx,d\ty, d\tw) - \bar \mu_s^{\zeta,\tz} (d\tx,d\ty, d\tw) \right) d\tz  \right| \\ 
 ~~~~~~~~~~ + \displaystyle \left|  \inttt \Lambda( \xi, \zeta,Z [\bar \mu](s,\xi,\zeta,x,y,w), \tx, \ty, \tw) \left(\mu_s^\xztt(d\tx,d\ty,d\tw) - \bar \mu_s^\xztt(d\tx,d\ty,d\tw) \right) d \txi d \tz \right| \Bigg) ds\\
 \end{array}$$
By Proposition~\ref{prop:compacite}, for all $s\in[0,T]$, for a.e. $\xi,\zeta \in I$,  the measures $\mu_s^{\xi,\zeta}$ and
$\bar\mu_s^{\xi,\zeta}$ are supported in $B_T^{\phi,\Lambda}(R_X,R_Y,R_M)$, hence
$\chi_T\equiv 1$ on these supports and we can rewrite the above term as 
 $$\begin{array}{l} 
 \dsp \int_0^t \left| F_\Lambda[\mu_s](\xi,\zeta,Z [\bar \mu](s,\xi,\zeta,x,y,w))  - F_\Lambda[ \bar \mu_s](\xi,\zeta,Z [\bar \mu](s,\xi,\zeta,x,y,w)) \right| ds \\
 = \displaystyle \int_0^t \Bigg( \left| \intt \tw \chi_T(\tx,\ty, \tw) \phi(s,X [\bar \mu](s,\xi,\zeta,x,y,w),\ty)  \left(\mu_s^{\xi,\tz} (d\tx,d\ty, d\tw) - \bar \mu_s^{\xi,\tz} (d\tx,d\ty, d\tw) \right) d\tz\right|  \\ 
 ~~~~~~ + \displaystyle \left|  \intt \tw \chi_T(\tx,\ty, \tw) \phi(s,Y [\bar \mu](s,\xi,\zeta,x,y,w),\ty) \left( \mu_s^{\zeta,\tz} (d\tx,d\ty, d\tw) - \bar \mu_s^{\zeta,\tz} (d\tx,d\ty, d\tw) \right) d\tz  \right| \\ 
 ~~~~~~ + \displaystyle \left|  \inttt \chi_T(\tx,\ty, \tw) \Lambda( \xi, \zeta,Z [\bar \mu](s,\xi,\zeta,x,y,w), \tx, \ty, \tw) \left(\mu_s^\xztt(d\tx,d\ty,d\tw) - \bar \mu_s^\xztt(d\tx,d\ty,d\tw) \right) d \txi d \tz \right| \Bigg) ds.
 \end{array}$$
 We introduced $\chi_T$ so that the integrands become globally bounded and Lipschitz, hence admissible as test functions in the definition of  $d_{\rm BL}$. Moreover, by Proposition~\ref{prop:compacite}, we have
\[
Z[\bar\mu](s,\xi,\zeta,x,y,w)\in B_T^{\phi,\Lambda}(R_X,R_Y,R_M)
\subset \tilde B_T^{\phi,\Lambda}(R_X,R_Y,R_M),
\]
which allows us to apply Lemma~\ref{lem:normes_BL} with
$(x,y,w)=Z[\bar\mu](s,\xi,\zeta,x,y,w)$ and obtain
  $$\begin{array}{l} 
 \dsp \int_0^t \left| F_\Lambda[\mu_s](\xi,\zeta,Z [\bar \mu](s,\xi,\zeta,x,y,w))  - F_\Lambda[ \bar \mu_s](\xi,\zeta,Z [\bar \mu](s,\xi,\zeta,x,y,w)) \right| ds \\
 \leq\displaystyle \int_0^t \Bigg( L_3 \left| \intt  \frac{1}{L_3} \tw \chi_T(\tx,\ty, \tw) \phi(s,X [\bar \mu](s,\xi,\zeta,x,y,w),\ty)  \left(\mu_s^{\xi,\tz} (d\tx,d\ty, d\tw) - \bar \mu_s^{\xi,\tz} (d\tx,d\ty, d\tw) \right) d\tz\right|  \\ 
  + \displaystyle L_3\left| \intt \frac{1}{L_3} \tw \chi_T(\tx,\ty, \tw) \phi(s,Y [\bar \mu](s,\xi,\zeta,x,y,w),\ty) \left( \mu_s^{\zeta,\tz} (d\tx,d\ty, d\tw) - \bar \mu_s^{\zeta,\tz} (d\tx,d\ty, d\tw) \right) d\tz  \right| \\ 
 + \displaystyle L_2 \left|  \inttt  \frac{1}{L_2}\chi_T(\tx,\ty, \tw) \Lambda( \xi, \zeta,Z [\bar \mu](s,\xi,\zeta,x,y,w), \tx, \ty, \tw) \left(\mu_s^\xztt(d\tx,d\ty,d\tw) - \bar \mu_s^\xztt(d\tx,d\ty,d\tw) \right) d \txi d \tz \right| \Bigg) ds\\
 \dsp \leq   L_3\int_0^t \left( \int_I  d_{\rm BL}(\mu_s^{\xi, \tz}(\cdot), \bar \mu_s^{\xi, \tz}(\cdot)) d\tz \right) ds+ L_3\int_0^t \left( \int_I  d_{\rm BL}(\mu_s^{\zeta, \tz}(\cdot), \bar \mu_s^{\zeta, \tz}(\cdot)) d\tz \right) ds \\
 ~~~~~~~~~~~~~ ~~~~~~~~~~~~~ ~~~~~~~~~~~~~ ~~~~~~~~~~~~~ ~~~~~~~~~~~~~ ~~~~~~~~~~~~~ + L_2  \int_0^t \left( \int_{I^2}  d_{\rm BL}(\mu_s^{\tilde \xi , \tz}(\cdot), \bar \mu_s^{\tilde \xi , \tz}(\cdot))  d \tilde \xi d\tz  \right) ds %\\
% \dsp \leq ( 2 L_3+ L_2) \int_0^t  \esssup_{(\tilde \xi, \tz) \in I^2}  d_{\rm BL}(\mu_s^{\tilde \xi, \tz}(\cdot), \bar \mu_s^{\tilde \xi, \tz}(\cdot) ds  
 \end{array}$$
 Thus, we have
$$\begin{array}{l}
 \left| Z[\mu](t,\xi,\zeta,x,y,w) - Z[\bar \mu](t,\xi,\zeta,x,y,w)\right|  \\
 \dsp \leq \int_0^t C_1(T)  \left| Z[\mu](s,\xi,\zeta,x,y,w) - Z[\bar \mu](s,\xi,\zeta,x,y,w)\right| ds +  L_3\int_0^t \left( \int_I  d_{\rm BL}(\mu_s^{\xi, \tz}(\cdot), \bar \mu_s^{\xi, \tz}(\cdot)) d\tz \right) ds  \\
 \dsp ~~~~~~~~L_3\int_0^t \left( \int_I  d_{\rm BL}(\mu_s^{\zeta, \tz}(\cdot), \bar \mu_s^{\zeta, \tz}(\cdot)) d\tz \right) ds+ L_2  \int_0^t \left( \int_{I^2}  d_{\rm BL}(\mu_s^{ \tilde \xi , \tz}(\cdot), \bar \mu_s^{\tilde \xi , \tz}(\cdot))  d \tilde \xi d\tz  \right) ds  \\
 \dsp \leq C_2(T) \left( \int_0^t  \left| Z[\mu](s,\xi,\zeta,x,y,w) - Z[\bar \mu](s,\xi,\zeta,x,y,w)\right| ds +  \int_0^t \left( \int_I  d_{\rm BL}(\mu_s^{\xi, \tz}(\cdot), \bar \mu_s^{\xi, \tz}(\cdot)) d\tz \right) ds \right. \\
  \dsp ~~~~~~~~~~~~~~~~~~~~~~~~~~~~~~+ \int_0^t \left( \int_I  d_{\rm BL}(\mu_s^{\zeta, \tz}(\cdot), \bar \mu_s^{\zeta, \tz}(\cdot)) d\tz \right) ds+ \left. \int_0^t \left( \int_{I^2}  d_{\rm BL}(\mu_s^{\tilde \xi , \tz}(\cdot), \bar \mu_s^{\tilde \xi , \tz}(\cdot))  d \tilde \xi d\tz  \right) ds \right)
  \end{array}$$
  with $C_2(T):= \max(C_1(T),  L_3, L_2)$. Then, using Gronwall's lemma, we get
  $$\begin{array}{l}
 \left| Z[\mu](t,\xi,\zeta,x,y,w) - Z[\bar \mu](t,\xi,\zeta,x,y,w)\right|  \\
 \dsp \leq C_2(T) \int_0^t e^{C_2(T) (t-s)}    \left( \int_I  d_{\rm BL}(\mu_s^{\xi, \tz}(\cdot), \bar \mu_s^{\xi, \tz}(\cdot)) d\tz + \int_I  d_{\rm BL}(\mu_s^{\zeta, \tz}(\cdot), \bar \mu_s^{\zeta, \tz}(\cdot)) d\tz +\int_{I^2}  d_{\rm BL}(\mu_s^{\tilde \xi , \tz}(\cdot), \bar \mu_s^{ \tilde \xi , \tz}(\cdot))  d \tilde \xi d\tz  \right) ds.
  \end{array}$$
Thus, integrating over $\R^d \times \R^d \times \R$ against  $\mu_*^{\xi,\zeta}$,  we get that 
 $$ \begin{array}{l}
 d_{\rm BL}(\mathcal{F}[\mu]_t^{\xi,\zeta},\mathcal{F}[\bar \mu]_t^{\xi,\zeta}) \\
 \dsp \leq C_2(T) \int_0^t e^{C_2(T) (t-s)}   \left( \int_I  d_{\rm BL}(\mu_s^{\xi, \tz}(\cdot), \bar \mu_s^{\xi, \tz}(\cdot)) d\tz + \int_I  d_{\rm BL}(\mu_s^{\zeta, \tz}(\cdot), \bar \mu_s^{\zeta, \tz}(\cdot)) d\tz  + \int_{I^2}  d_{\rm BL}(\mu_s^{\tilde \xi , \tz}(\cdot), \bar \mu_s^{\tilde \xi , \tz}(\cdot))  d \tilde \xi d\tz  \right) ds.
  \end{array}$$
  Thus, integrating over $\xi,\zeta$, we have
   $$ \begin{array}{l}
\dsp\int_{I^2}    d_{\rm BL}(\mathcal{F}[\mu]_t^{\xi,\zeta},\mathcal{F}[\bar \mu]_t^{\xi,\zeta})  d\xi d\zeta \\
 \dsp \leq C_2(T) \int_0^t e^{C_2(T) (t-s)} 3 \left(  \int_{I^2}  d_{\rm BL}(\mu_s^{ \xi , \zeta}(\cdot), \bar \mu_s^{  \xi , \zeta}(\cdot))  d \xi d\zeta  \right) ds.
  \end{array}$$
  where we have used Fubini's theorem and a suitable renaming of dummy variables to make appear the multiplicative factor~$3$.

  Therefore, for $\alpha >0$, we have 
     $$ \begin{array}{l}
\dsp e^{-\alpha t}  \dsp\int_{I^2}  d_{\rm BL}(\mathcal{F}[\mu]_t^{\xi,\zeta},\mathcal{F}[\bar \mu]_t^{\xi,\zeta})  d\xi d\zeta \\
 \dsp \leq 3 e^{(C_2(T) - \alpha) t} \int_0^t e^{(\alpha - C_2(T))s}   C_2(T) e^{- \alpha s }  d_{1}( \mu_s, \bar \mu_s) ds\\
 \dsp \leq  3 C_2(T)  \left( \int_0^t e^{(\alpha-C_2(T))(s-t)} ds \right) \, \left( \sup_{\tau \in [0,T]} e^{- \alpha \tau }  d_{1}( \mu_\tau, \bar \mu_\tau) \right)\\
 \dsp \leq \frac{3 C_2(T)}{\alpha - C_2(T)}  d_{1}^\alpha( \mu, \bar \mu)
  \end{array}$$
Thus, if   $\frac{3 C_2(T)}{\alpha - C_2(T)}  <1$ i.e. $\alpha > 4 C_2(T)$, we conclude that it is contracting by taking the supremum on $t \in [0,T]$ on the left-hand side and we can apply the fixed point theorem which concludes the proof. 
   \end{proof}

\section{Stability estimates for the Vlasov-type equation} \label{sec:stability}
In this section, we investigate a structural property of the Vlasov-type equation that plays a central role in our analysis, namely the stability of solutions with respect to both the initial data and the weight dynamics $\Lambda$. Such stability estimates are of particular importance, as the main mean-field limit result established in Section \ref{sec:mfl} will fundamentally rely on them.
\begin{theorem}\label{theo:stability} Let the interaction function $\phi$  satisfy Hypothesis \ref{hyp:phi}. Let the weight dynamics $\Lambda$, $\bar \Lambda$ satisfy Hypotheses \ref{hyp:psi} and \ref{hyp:struct_lambda}.
Let the initial data $\mu_0$, $\bar \mu_0 $  satisfy Hypothesis \ref{hyp:supp_compact}  with $\nu$ the Lebesgue measure. Let $\mu$, $\bar \mu \in \mathcal{C}([0,T], \mathcal{P}_{1,\nu}(I^2 \times \mathbb{R}^d \times \mathbb{R}^d \times \mathbb{R}))$ be the unique distributional solution to \eqref{eq:vlasov-equation}-\eqref{eq:vlasov-force} issued at $\mu_0$ with given $\Lambda$  (respectively $\bar \mu_0$ with given $\bar \Lambda$). Then, we have 
      $$ \begin{array}{l}
    d_1( \mu_t, \bar \mu_t) \\  \leq  e^{4 \,  C_3(T) t}  \Bigg( d_{1}( \mu_0, \bar \mu_0) \\
~+\dsp  C_3(T)   \int_0^t \int_{{I^2} \times \R^d \times \R^d \times \R}   \inttt |\Lambda - \bar \Lambda| ( \xi, \zeta, x,y,w, \tx, \ty, \tw) \bar \mu_s^\xztt(d\tx,d\ty,d\tw)d \txi d \tz \mu_0^{\xi,\zeta}(dx,dy,dw) {d \xi d\zeta} ds \Bigg)
     \end{array}$$
     with $C_3(T) := \max(C_2(T),1)$.
\end{theorem}

As a preliminary step towards the proof of Theorem \ref{theo:stability}, we establish the next lemma.
\begin{lemma}\label{lem:phiZ_lip}
For $(\xi, \zeta) \in I^2$, for $\varphi \in \text{BL}(\R^d \times \R^d \times \R)$ with $\|\varphi \Vert_{\rm BL} \leq 1$, we have $$\|\varphi \circ Z[\bar \mu](t,\xi,\zeta, \cdot) \Vert_{\rm BL}  \leq  e^{{C}_1(T) \, t}.$$
\end{lemma}

\begin{proof}
For $(x_1,y_1,w_1), (x_2,y_2,w_2)  \in \mathbb{R}^d \times \mathbb{R}^d \times \mathbb{R}$, we have
$$\begin{array}{l}
 \left| Z[\bar \mu](t,\xi,\zeta,x_1,y_1,w_1) - Z[\bar \mu](t,\xi,\zeta,x_2,y_2,w_2)\right|  \\
\dsp = \left|  \left( \begin{array}{c}
x_1\\ 
y_1\\ 
w_1
 \end{array} \right)  +  \int_0^t  F_{\bar \Lambda}[\bar \mu_s](\xi,\zeta,Z [\bar \mu](s,\xi,\zeta,x_1,y_1,w_1))  ds -  \left( \begin{array}{c}
x_2\\ 
y_2\\ 
w_2
 \end{array} \right) -  \int_0^t  F_{\bar \Lambda}[\bar \mu_s](\xi,\zeta,Z [\bar \mu](s,\xi,\zeta,x_2,y_2,w_2))  ds  \right|\\
\dsp \leq \left(|x_1-x_2| + |y_1-y_2| + |w_1-w_2| \right) +  \int_0^t C_1(T) \left| Z[\bar \mu](s,\xi,\zeta,x_1,y_1,w_1) - Z[\bar \mu](s,\xi,\zeta,x_2,y_2,w_2)\right|  ds
\end{array}$$
applying Lemma \ref{lem:Lip_force}. Thus, using Gronwall's lemma, we deduce that 
$$\begin{array}{l}
\dsp  \left| Z[\bar \mu](t,\xi,\zeta,x_1,y_1,w_1) - Z[\bar \mu](t,\xi,\zeta,x_2,y_2,w_2)\right|   \leq \left(|x_1-x_2| + |y_1-y_2| + |w_1-w_2| \right) e^{C_1(T) \,t}.
\end{array}$$
Since $\|\varphi\|_{\rm BL}\le 1$, we have $\|\varphi\|_{L^\infty}\le 1$ and
$[\varphi]_{\rm Lip}\le 1$. Therefore,
\[
\|\varphi\circ Z[\bar\mu](t,\xi,\zeta,\cdot)\|_{L^\infty}
\le \|\varphi\|_{L^\infty}\le 1,
\]
and for all $(x_1,y_1,w_1),(x_2,y_2,w_2)\in\R^d\times\R^d\times\R$,
\[ \begin{array}{rcl}
|\varphi(Z[\bar\mu](t,\xi,\zeta,x_1,y_1,w_1))
-\varphi(Z[\bar\mu](t,\xi,\zeta,x_2,y_2,w_2))|
& \le  & \left| Z[\bar \mu](t,\xi,\zeta,x_1,y_1,w_1) - Z[\bar \mu](t,\xi,\zeta,x_2,y_2,w_2)\right|\\
&  \leq &  e^{C_1(T)t}\bigl(|x_1-x_2|+|y_1-y_2|+|w_1-w_2|\bigr).
\end{array}
\]
Hence,
\[
\|\varphi\circ Z[\bar\mu](t,\xi,\zeta,\cdot)\|_{\rm BL}
\le e^{C_1(T)t},
\]
which concludes the proof.
\end{proof}

\begin{proof}[Proof of Theorem \ref{theo:stability}]
We pick $\mu$ and $\bar \mu$ two solutions with respective initial data $\mu_0$ and $\bar \mu_0$ and respective functions $\Lambda$ and $\bar \Lambda$ associated to the weight dynamics. We have 
$$ \begin{array}{l}
d_{\rm BL}(\mu_t^{\xi,\zeta}, \bar \mu_t^{\xi,\zeta} ) \\
\dsp \leq \sup_{\Vert \varphi\Vert_{\rm BL}\leq 1}\int_{\R^d \times \R^d \times \R}\varphi(x,y,w)\left(d\mu_t^{\xi,\zeta}(x,y,w)-d\bar \mu_t^{\xi,\zeta}(x,y,w)\right)\\
\dsp = \sup_{\Vert \varphi\Vert_{\rm BL}\leq 1}\int_{\R^d \times \R^d \times \R}\varphi\left(Z[\mu](t,\xi,\zeta,x,y,w)\right) d\mu_0^{\xi,\zeta}(x,y,w) - \int_{\R^d \times \R^d \times \R}\varphi\left (Z[\bar \mu](t,\xi,\zeta,x,y,w)\right) d\bar \mu_0^{\xi,\zeta}(x,y,w)\\
\dsp = \sup_{\Vert \varphi\Vert_{\rm BL}\leq 1}\Bigg(  \int_{\R^d \times \R^d \times \R} \left( \varphi\left(Z[\mu](t,\xi,\zeta,x,y,w)\right)  - \varphi\left(Z[\bar \mu](t,\xi,\zeta,x,y,w)\right) \right)d\mu_0^{\xi,\zeta}(x,y,w) \\
\dsp ~~~~~~~~~~~~~~~~~~~~~~~~~ + \int_{\R^d \times \R^d \times \R}\varphi\left (Z[\bar \mu](t,\xi,\zeta,x,y,w)\right) d(\mu_0^{\xi,\zeta}(x,y,w) - \bar \mu_0^{\xi,\zeta}(x,y,w)) \Bigg)\\
\dsp \leq \sup_{\Vert \varphi\Vert_{\rm BL}\leq 1}  \int_{\R^d \times \R^d \times \R} \left| Z[\mu](t,\xi,\zeta,x,y,w) - Z[\bar \mu](t,\xi,\zeta,x,y,w)\right|  d\mu_0^{\xi,\zeta}(x,y,w) \\
\dsp ~~~~~~~~~~~~~~~~~~~~~~~~~ +\|\varphi \circ Z[\bar \mu](t,\xi, \zeta, \cdot) \Vert_{\rm BL} d_{\rm BL}(\mu_0^{\xi,\zeta}, \bar \mu_0^{\xi,\zeta} )
\end{array}$$
We start by dealing with the first term. We know that 
$$\begin{array}{l}
 \left| Z[ \mu](t,\xi,\zeta,x,y,w) - Z[\bar \mu](t,\xi,\zeta,x,y,w)\right|  \\
\dsp = \left|   \int_0^t  F_{\Lambda}[ \mu_s](\xi,\zeta,Z [\mu](s,\xi,\zeta,x,y,w))  ds -  \int_0^t  F_{\bar \Lambda}[\bar \mu_s](\xi,\zeta,Z [\bar \mu](s,\xi,\zeta,x,y,w))  ds  \right|\\
\dsp \leq \left|   \int_0^t  F_{\Lambda}[ \mu_s](\xi,\zeta,Z [\mu](s,\xi,\zeta,x,y,w))  ds -  \int_0^t  F_{\Lambda}[\mu_s](\xi,\zeta,Z [\bar \mu](s,\xi,\zeta,x,y,w))  ds  \right|\\
\dsp +  \left|   \int_0^t  F_{\Lambda}[ \mu_s](\xi,\zeta,Z [\bar \mu](s,\xi,\zeta,x,y,w))  ds -  \int_0^t  F_{\Lambda}[\bar \mu_s](\xi,\zeta,Z [\bar \mu](s,\xi,\zeta,x,y,w))  ds  \right|\\
\dsp +  \left|   \int_0^t  F_{\Lambda}[ \bar \mu_s](\xi,\zeta,Z [\bar \mu](s,\xi,\zeta,x,y,w))  ds -  \int_0^t  F_{\bar \Lambda}[\bar \mu_s](\xi,\zeta,Z [\bar \mu](s,\xi,\zeta,x,y,w))  ds  \right| =: A_1 +A_2 +A_3.
\end{array}$$
Using Lemma \ref{lem:Lip_force}, we can bound the first term as follows
$$\begin{array}{l}
|A_1| \leq\dsp  \int_0^t C_1(T)  \left| Z[ \mu](s,\xi,\zeta,x,y,w) - Z[\bar \mu](s,\xi,\zeta,x,y,w)\right| ds.
\end{array}$$
Besides, we already have bounded $|A_2|$ in the proof of Propositon \ref{theo:well-posedness-vlasov} and obtained
$$\begin{array}{l}
|A_2| \dsp  \dsp \leq   L_3\int_0^t \left( \int_I  d_{\rm BL}(\mu_s^{\xi, \tz}(\cdot), \bar \mu_s^{\xi, \tz}(\cdot)) d\tz \right) ds + L_3\int_0^t \left( \int_I  d_{\rm BL}(\mu_s^{\zeta, \tz}(\cdot), \bar \mu_s^{\zeta, \tz}(\cdot)) d\tz \right) ds \\
~~~~~~~~~~~~~~~~~~~~~~~~~~~~~~ + L_2  \int_0^t \left( \int_{I^2}  d_{\rm BL}(\mu_s^{\tilde \xi , \tz}(\cdot), \bar \mu_s^{\tilde \xi , \tz}(\cdot))  d \tilde \xi d\tz  \right) ds.
\end{array}$$ 
Therefore, we deduce that 
$$\begin{array}{l}
 \left| Z[ \mu](t,\xi,\zeta,x,y,w) - Z[\bar \mu](t,\xi,\zeta,x,y,w)\right|\\   
 \displaystyle \leq ~\int_0^t C_1(T)  \left| Z[ \mu](s,\xi,\zeta,x,y,w) - Z[\bar \mu](s,\xi,\zeta,x,y,w)\right| ds\\
\displaystyle ~+   L_3\int_0^t \left( \int_I  d_{\rm BL}(\mu_s^{\xi, \tz}(\cdot), \bar \mu_s^{\xi, \tz}(\cdot)) d\tz \right) ds + L_3\int_0^t \left( \int_I  d_{\rm BL}(\mu_s^{\zeta, \tz}(\cdot), \bar \mu_s^{\zeta, \tz}(\cdot)) d\tz \right) ds  \\
~~~~~~~~~~~~~~~~~~~~~~~~~\dsp + L_2  \int_0^t \left( \int_{I^2}  d_{\rm BL}(\mu_s^{\tilde \xi , \tz}(\cdot), \bar \mu_s^{\tilde \xi , \tz}(\cdot))  d \tilde \xi d\tz  \right) ds\\
~+ \displaystyle \int_0^t  \inttt |\Lambda - \bar \Lambda| ( \xi, \zeta, x,y,w, \tx, \ty, \tw) \bar \mu_s^\xztt(d\tx,d\ty,d\tw)d \txi d \tz ds \\
\dsp \leq  C_2(T)  \int_0^t \Bigg( \left| Z[ \mu](s,\xi,\zeta,x,y,w) - Z[\bar \mu](s,\xi,\zeta,x,y,w)\right|   + \int_I  d_{\rm BL}(\mu_s^{\xi, \tz}(\cdot), \bar \mu_s^{\xi, \tz}(\cdot)) d\tz  \\
~~~~~~~~~~~~~~~~~~~~~~~~~\dsp +  \int_I  d_{\rm BL}(\mu_s^{\zeta, \tz}(\cdot), \bar \mu_s^{\zeta, \tz}(\cdot)) d\tz   +  \int_{I^2}  d_{\rm BL}(\mu_s^{\tilde \xi , \tz}(\cdot), \bar \mu_s^{\tilde \xi , \tz}(\cdot))  d \tilde \xi d\tz  \Bigg) ds\\
~+ \displaystyle \int_0^t  \inttt | \Lambda - \bar \Lambda|( \xi, \zeta, x,y,w, \tx, \ty, \tw) \bar \mu_s^\xztt(d\tx,d\ty,d\tw)d \txi d \tz ds.
 \end{array}$$
 Thus, integrating over $\R^d \times \R^d \times \R$ against  $\mu_0^{\xi,\zeta}$,  we get
 $$\begin{array}{l}
 \dsp \int_{\R^d \times \R^d \times \R}  \left| Z[ \mu](t,\xi,\zeta,x,y,w) - Z[\bar \mu](t,\xi,\zeta,x,y,w)\right| \mu_0^{\xi,\zeta}(dx,dy,dw)\\
 \dsp \leq  C_2(T)  \int_0^t  \int_{\R^d \times \R^d \times \R}  \left| Z[ \mu](s,\xi,\zeta,x,y,w) - Z[\bar \mu](s,\xi,\zeta,x,y,w)\right| \mu_0^{\xi,\zeta}(dx,dy,dw)\\
~\dsp +    C_2(T)  \int_0^t   \left(\int_I  d_{\rm BL}(\mu_s^{\xi, \tz}(\cdot), \bar \mu_s^{\xi, \tz}(\cdot)) d\tz   + \int_I  d_{\rm BL}(\mu_s^{\zeta, \tz}(\cdot), \bar \mu_s^{\zeta, \tz}(\cdot)) d\tz +  \int_{I^2}  d_{\rm BL}(\mu_s^{\tilde \xi , \tz}(\cdot), \bar \mu_s^{\tilde \xi , \tz}(\cdot)) d \tilde \xi d\tz \right)  ds\\
~+ \displaystyle  \int_0^t \int_{\R^d \times \R^d \times \R}   \inttt | \Lambda - \bar \Lambda|( \xi, \zeta, x,y,w, \tx, \ty, \tw) \bar \mu_s^\xztt(d\tx,d\ty,d\tw)d \txi d \tz \mu_0^{\xi,\zeta}(dx,dy,dw) ds.
 \end{array}$$
 Thus, applying Gronwall's lemma, we get 
  $$\begin{array}{l}
 \dsp \int_{\R^d \times \R^d \times \R}  \left| Z[ \mu](t,\xi,\zeta,x,y,w) - Z[\bar \mu](t,\xi,\zeta,x,y,w)\right| \mu_0^{\xi,\zeta}(dx,dy,dw)\\
\leq \dsp  C_3(T)  \int_0^t e^{C_3(T) (t-s)}   \Bigg(\int_I  d_{\rm BL}(\mu_s^{\xi, \tz}(\cdot), \bar \mu_s^{\xi, \tz}(\cdot)) d\tz   + \int_I  d_{\rm BL}(\mu_s^{\zeta, \tz}(\cdot), \bar \mu_s^{\zeta, \tz}(\cdot)) d\tz  +  \int_{I^2}  d_{\rm BL}(\mu_s^{ \tilde \xi , \tz}(\cdot), \bar \mu_s^{ \tilde \xi , \tz}(\cdot)) d \tilde \xi d\tz  ds\\
~+ \displaystyle  \int_{\R^d \times \R^d \times \R}   \inttt |\Lambda - \bar \Lambda|( \xi, \zeta, x,y,w, \tx, \ty, \tw) \bar \mu_s^\xztt(d\tx,d\ty,d\tw)d \txi d \tz \mu_0^{\xi,\zeta}(dx,dy,dw) \Bigg) ds
  \end{array}$$
  with $C_3(T) = \max(C_2(T),1)$. From that, we deduce that 
  $$ \begin{array}{l}
d_{\rm BL}(\mu_t^{\xi,\zeta}, \bar \mu_t^{\xi,\zeta} ) \\
\dsp \leq \|\varphi \circ  Z[\bar \mu](t,\xi,\zeta,\cdot) \Vert_{\rm BL} d_{\rm BL}(\mu_0^{\xi,\zeta}, \bar \mu_0^{\xi,\zeta} )\\
~ + \dsp  C_3(T)  \int_0^t e^{C_3(T) (t-s)}   \Bigg(\int_I  d_{\rm BL}(\mu_s^{\xi, \tz}(\cdot), \bar \mu_s^{\xi, \tz}(\cdot)) d\tz  + \int_I  d_{\rm BL}(\mu_s^{\zeta, \tz}(\cdot), \bar \mu_s^{\zeta, \tz}(\cdot)) d\tz   +  \int_{I^2}  d_{\rm BL}(\mu_s^{\tilde \xi , \tz}(\cdot), \bar \mu_s^{\tilde \xi , \tz}(\cdot)) d \tilde \xi d\tz  ds\\
~+ \displaystyle  \int_{\R^d \times \R^d \times \R}   \inttt | \Lambda - \bar \Lambda|( \xi, \zeta, x,y,w, \tx, \ty, \tw) \bar \mu_s^\xztt(d\tx,d\ty,d\tw)d \txi d \tz \mu_0^{\xi,\zeta}(dx,dy,dw) \Bigg) ds\\
\dsp \leq e^{C_1(T)t}d_{\rm BL}(\mu_0^{\xi,\zeta}, \bar \mu_0^{\xi,\zeta} )\\
~ + \dsp  C_3(T)  \int_0^t e^{C_3(T) (t-s)}   \Bigg(\int_I  d_{\rm BL}(\mu_s^{\xi, \tz}(\cdot), \bar \mu_s^{\xi, \tz}(\cdot)) d\tz  + \int_I  d_{\rm BL}(\mu_s^{\zeta, \tz}(\cdot), \bar \mu_s^{\zeta, \tz}(\cdot)) d\tz   +  \int_{I^2}  d_{\rm BL}(\mu_s^{ \tilde \xi , \tz}(\cdot), \bar \mu_s^{\tilde \xi , \tz}(\cdot)) d \tilde \xi d\tz  ds\\
~+ \displaystyle  \int_{\R^d \times \R^d \times \R}   \inttt |\Lambda - \bar \Lambda|( \xi, \zeta, x,y,w, \tx, \ty, \tw) \bar \mu_s^\xztt(d\tx,d\ty,d\tw)d \txi d \tz \mu_0^{\xi,\zeta}(dx,dy,dw) \Bigg) ds
  \end{array}$$
  using Lemma \ref{lem:phiZ_lip}. We integrate in $\xi,\zeta$ and get
    $$ \begin{array}{l}
    d_{1}( \mu_t, \bar \mu_t)\\ 
   \leq\dsp   e^{C_3(T)t}   d_{1}( \mu_0, \bar \mu_0) + \dsp 3 \,  C_3(T)  \int_0^t e^{C_3(T) (t-s)}     d_{1}( \mu_s, \bar \mu_s)\\
+ \dsp   C_3(T)  \int_0^t e^{C_3(T) (t-s)}   \int_{{I^2} \times \R^d \times \R^d \times \R}   \inttt |\Lambda - \bar \Lambda|( \xi, \zeta, x,y,w, \tx, \ty, \tw) \bar \mu_s^\xztt(d\tx,d\ty,d\tw)d \txi d \tz \mu_0^{\xi,\zeta}(dx,dy,dw) \\
~~~~~~~~~~~~~~~~~~ ~~~~~~~~~ ~~~~~~~~~~~~~ ~~~~~~~~~~~~~ ~~~~~~~~~~~~~ ~~~~~~~~~ ~~~~~~~~~ ~~~~~~~~~ ~~~~~~~~~~~~~ ~~~~~~~~~~~~~ ~~~~~~~~~~~~~ ~ { d \xi d \zeta} ds
     \end{array}$$
     since $C_3(T)\ge C_1(T)$.
     We multiply by $ e^{- C_3(T) t}$ and apply Gronwall's lemma to obtain
         $$ \begin{array}{l}
e^{- C_3(T) t}   d_{1}( \mu_t, \bar \mu_t) \\  \leq  e^{3 \,  C_3(T) t} \Bigg(d_{1}( \mu_0, \bar \mu_0)\\
~+\dsp  C_3(T)   \int_0^t \int_{{I^2} \times \R^d \times \R^d \times \R}   \inttt | \Lambda - \bar \Lambda|( \xi, \zeta, x,y,w, \tx, \ty, \tw) \bar \mu_s^\xztt(d\tx,d\ty,d\tw)d \txi d \tz \mu_0^{\xi,\zeta}(dx,dy,dw)  {d \xi d \zeta}  ds \Bigg)
     \end{array}$$
     from which we deduce 
      $$ \begin{array}{l}
    d_{1}( \mu_t, \bar \mu_t) \\  \leq  e^{4 \,  C_3(T) t} \Bigg(d_{1}( \mu_0, \bar \mu_0) \\
~+\dsp  C_3(T)   \int_0^t \int_{{I^2}  \times \R^d \times \R^d \times \R}   \inttt | \Lambda - \bar \Lambda|( \xi, \zeta, x,y,w, \tx, \ty, \tw) \bar \mu_s^\xztt(d\tx,d\ty,d\tw)d \txi d \tz \mu_0^{\xi,\zeta}(dx,dy,dw) { d \xi d \zeta} ds \Bigg)
     \end{array}$$
\end{proof}

\section{Propagation of independance} \label{sec:independance}
\subsection{An intermediate particle system}

The main objective of this work is to rigorously establish the mean-field limit of the particle system \eqref{eq:particle_syst_restrictive} towards the Vlasov-type equation \eqref{eq:vlasov-equation}-\eqref{eq:vlasov-force}. To this end, we adopt an approach inspired by \cite{Snitzman} which consists in introducing an auxiliary system expected to approximate the original $N$-particle dynamics as  $N$ goes to $\infty$ under suitable assumptions. In this framework, we make use of the notion of propagation of independence introduced in \cite{JabinPoyatoSoler21}. Indeed, the auxiliary system is initialized with the family of independent variables  $\{X^{N,0}, W^{N,0}\}$ where $X^{N,0} := (X_i^{N,0})_{i \in \{1, \dots, N\}}$ and $W^{N,0} := (W_{ij}^{N,0})_{i,j \in \{1, \dots, N\}}$. 
In order to effectively propagate this independence property, however, we shall impose an additional restrictive assumption on $\Lambda$ namely that it no longer depends on the spatial variable $x$.
Thus, we will adopt this assumption for the rest of this section.

\begin{rem}
Although restrictive, this assumption remains meaningful from a modeling perspective. In particular, it naturally encompasses the study of interactions on online social platforms in connection with moderation mechanisms such as deplatforming, where users are banned once they start expressing extremist opinions (see \cite{thomas2023disrupting}), or shadow banning, which consists in reducing the visibility of certain extremist users (see \cite{chen2024shadow}). In the latter case, the reduction applies uniformly to all, or to a subset of, other users, in a way that is independent of their own opinions and depends solely on that of the shadow-banned individual. The model considered in this work, even under this restrictive assumption, is therefore sufficiently rich to capture these scenarios, which are both complex and of significant interest.
\end{rem}
The particle system to be studied is now rewritten under this restrictive assumption.
 \begin{align} \label{eq:particle_syst_restrictive}
\left\{ \begin{aligned}
 \frac{d}{dt} X_i^N & =  \frac{1}{N} \sum_{j=1}^N W_{ij}^N \phi(t,X_i^N,X_j^N)\\
  \frac{d}{dt}  W_{ij}^N  & =  \frac{1}{N^2} \sum_{\substack{i_1=1 \\ i _1\neq i}}^N  \sum_{\substack{j_1=1 \\ j _1\neq i}}^N  \Lambda_{ij}(X_j^N,W_{ij}^N,X_{i_1}^N,X_{j_1}^N,W_{{i_1}{j_1}}^N)
 \end{aligned} \right.
 \end{align}

\begin{rem}
Because the dependence on the first variable disappears, we must explicitly require $i_1,j_1 \neq i$, since we can no longer rely on Hypothesis~\ref{hyp:struct_lambda} to guarantee the absence of $X_i$ among the remaining variables. 
\end{rem}
We can now introduce the intermediate particle system that will serve as the starting point of our analysis and whose dynamics are governed by the following equations:
 \begin{align} \label{eq:intermediate1}
\left\{ \begin{aligned}
 \frac{d}{dt} \bar X_i^N & =  \frac{1}{N} \sum_{j=1}^N \E_i \croch{\bar W_{ij}^N \phi(t,\bar X_i^N,\bar X_j^N)}\\
  \frac{d}{dt}  \bar W_{ij}^N  & =  \frac{1}{N^2} \sum_{\substack{i_1=1 \\ i_1\neq i}}^N \sum_{\substack{j_1=1 \\ j _1\neq i}}^N\E_{i,j} \croch{\Lambda_{ij}(\bar X_j^N,\bar W_{ij}^N,\bar X_{i_1}^N,\bar X_{j_1}^N,\bar W_{{i_1}{j_1}}^N)}
 \end{aligned} \right.
 \end{align}
 where $\E_i := \E \croch{\cdot | \F_i}$ and  $\E_{i,j} := \E \croch{\cdot | \F_{i,j}} $  denote the expectations conditioned respectively to the natural filtrations $\dsp \F_i(t) := \sigma \pare{\{\bar X_i^N (s) \}_{s \leq t}}$ and  $\F_{i,j}(t) := \sigma \pare{\{(\bar X_j^N (s),\bar W_{ij}^N(s) )\}_{s \leq t}}$.\\

 Let us denote $\bar \lambda_t^{N,i,j} (\cdot, \cdot, \cdot) := \text{law} \left((\bxi(t), \bxj(t), \bwij(t))\right)$. 
For the sake of simplicity of presentation, we assume that the law
$\bar\lambda_t^{N,i,j}$ admits a density with respect to the Lebesgue measure on $\R^d \times \R^d \times \R$.
The same arguments can be carried out in the general measure-valued setting,
at the price of replacing densities by the corresponding marginal measures.
  We know that the law of $(\bxj(t), \bwij(t))$ is given by $$\text{law} \left((\bxj(t), \bwij(t))\right) = \int_{\R^d}  \lij ( \tx, \cdot, \cdot) d\tx. $$
Besides, by independence, the law of $(\bxj, \bwij)$ knowing $\bxi$ is the law of $(\bxj, \bwij).$ Therefore, we have 
\begin{equation*}
\E_i\left[ \bwij \phi(t, \bxi, \bxj)\right] = \int_{\R^d \times \R}  \tw \phi (t, \bxi, \ty) \int_{\R^d} \lij ( \tx, \ty, \tw) d \tx d\ty d\tw.
\end{equation*}

Similarly, by independence of $(X_j,W_{ij})$ from $(X_{i_1},X_{j_1},W_{i_1j_1})$ for $j,i_1,j_1 \neq i$, $$\E_{i,j} \croch{\Lambda_{ij}(\bar X_j^N,\bar W_{ij}^N,\bar X_{i_1}^N,\bar X_{j_1}^N,\bar W_{{i_1}{j_1}}^N)} = \int_{\R^d \times \R^d \times \R}\Lambda_{ij}(\bar X_j^N,\bar W_{ij}^N, \tx, \ty, \tw) \lijun(\tx, \ty, \tw) d \tx d\ty d\tw.$$ 
Thus, system \eqref{eq:intermediate1} can rewrite 
 \begin{align} \label{eq:intermediate2}
\left\{ \begin{aligned}
 \frac{d}{dt} \bar X_i^N & =  \frac{1}{N} \sum_{j=1}^N \int_{\R^d \times \R}  \tw \phi (t, \bxi, \ty) \int_{\R^d} \lij ( \tx, \ty, \tw) d \tx d\ty d\tw\\
  \frac{d}{dt}  \bar W_{ij}^N  & =  \frac{1}{N^2} \sum_{\substack{i_1=1 \\ i_1\neq i}}^N \sum_{\substack{j_1=1 \\ j _1\neq i}}^N\ \int_{\R^d \times \R^d \times \R}\ {\Lambda_{ij}(\bar X_j^N,\bar W_{ij}^N, \tx, \ty, \tw)}  \lijun(\tx, \ty, \tw) d \tx d\ty d\tw.
 \end{aligned} \right.
 \end{align}

\subsection{Scaling assumptions for the graph}

In what follows, we focus on the way the graph is initialized. We adopt the following scaling:
\begin{hyp} \label{hyp:graph_initial}
We suppose that 
\begin{equation}
\sup_{i,j} |W_{ij}^{N,0}| = O(1).
\end{equation}
\end{hyp} 

\begin{rem}
Under this hypothesis, the assumptions on the initial weights satisfy the ones of \cite{JabinPoyatoSoler21} in the non-adaptive setting, where propagation of independence has been established.
\end{rem}

\begin{rem}We notice that the all-to-all connection, i.e. $W_{ij}^{N,0}=1$ for all $i,j$ fits the framework. 
\end{rem}
We initialize the graph of the intermediate particle systems with the same initial data, i.e. $\bar W_{ij}^{N,0} = W_{ij}^{N,0}$.

The well-posedness of both the initial and the intermediate particle systems is stated below.
\begin{lemma}[Well-posedness of the initial and intermediate particle systems]
\label{lem:wellposednessintermediateparticlesystem}
Assume that the interaction function $\phi$ satisfies Hypothesis~\ref{hyp:phi}
and that the weight dynamics satisfy Hypotheses~\ref{hyp:psi} and~\ref{hyp:struct_lambda}.
Let the particle systems \eqref{eq:particle_syst_restrictive} and
\eqref{eq:intermediate1} be initialized with random variables
$(X^{N,0},W^{N,0})$ such that the family
$\{X_i^{N,0},\, W_{ij}^{N,0} : 1\le i,j\le N\}$ is independent,
$\E|X_i^{N,0}|<+\infty$, and $W^{N,0}$ satisfies
Hypothesis~\ref{hyp:graph_initial}.
Then both systems admit a unique solution on $[0,T]$, pathwise and in law,
denoted respectively by
\[
(X^{N},W^{N}) \quad \text{and} \quad (\bar X^{N},\bar W^{N}),
\]
where
\[
X^{N} := (X_i^{N})_{1\le i\le N}, \qquad
W^{N} := (W_{ij}^{N})_{1\le i,j\le N},
\]
and
\[
\bar X^{N} := (\bar X_i^{N})_{1\le i\le N}, \qquad
\bar W^{N} := (\bar W_{ij}^{N})_{1\le i,j\le N}.
\]
\end{lemma}

\begin{rem} The proof of the well-posedness of the initial particle system can be found in \cite{Throm_2024}, as our assumptions fall within his framework. The proof for the intermediate particle system is an adaptation of the classical argument developed in \cite{Snitzman}. More precisely, it extends the original setting, which relies on uniform and static weights, to a more general case allowing for 	adaptive non-uniform weights. A preliminary discussion on the transition from the classical framework to the case of static but non-uniform weights can be found in \cite[Section 3]{JabinPoyatoSoler21}. In particular, since the particles are non-exchangeable, the main novelty lies in the fact that, for each $\bar X_i$, the associated McKean SDE cannot be closed in terms of its own law. In our case, a similar difficulty arises, further complicated by the need to take into account the law of the weights as well.
\end{rem}

 From the initial bounds, we can derive additional estimates.
\begin{prop} 
For all $t \geq 0$, consider the solutions to the particle systems \eqref{eq:particle_syst_restrictive} and \eqref{eq:intermediate1} introduced in Lemma \ref{lem:wellposednessintermediateparticlesystem}, then we have
\begin{equation} \label{eq:bound_weights1}
\sup_{i,j} |W_{ij}^{N}(t)| \leq (C + C_\Lambda t)  \exp(C_\Lambda t)
\end{equation}
and
\begin{equation} \label{eq:bound_weights2}
\sup_{i,j} |\bar W_{ij}^{N}(t)| \leq (C + C_\Lambda t ) \exp(C_\Lambda t)
\end{equation}
with $C$ a positive number. 
\end{prop}

\begin{proof}
We integrate the second equation of \eqref{eq:particle_syst_restrictive}  and use Hypothesis \ref{hyp:psi} to  obtain 
$$\left| W_{ij}^{N}(t)\right|  \leq \left| W_{ij}^{N,0}\right| +  C_\Lambda  \int_0^t {\left(1+\left| W_{ij}^{N}(s)\right|\right)} ds$$
which, thanks to Gronwall's lemma, leads to 
$$\left|W_{ij}^{N}(t)\right|  \leq \left( \left| W_{ij}^{N,0}\right| + C_\Lambda t\right) \exp(C_\Lambda t).$$
We conclude to the existence of a constant $C >0$ satisfying \eqref{eq:bound_weights1} using Hypothesis \ref{hyp:graph_initial}. A similar reasoning leads to \eqref{eq:bound_weights2}.
\end{proof}

\subsection{Control of the error}
This subsection is devoted to an explicit quantification of the error arising from the approximation of the original system by the intermediate system.
\begin{lemma}[Error estimate]\label{lem:error-estimate}
Assume that the interaction function satisfies Hypothesis \ref{hyp:phi} and the weight dynamics Hypotheses \ref{hyp:psi} and \ref{hyp:struct_lambda}. 
Let the particle systems \eqref{eq:particle_syst_restrictive} and
\eqref{eq:intermediate1} be initialized with random variables
$(X^{N,0},W^{N,0})$ such that the family
$\{X_i^{N,0},\, W_{ij}^{N,0} : 1\le i,j\le N\}$ is independent,
$\E|X_i^{N,0}|<+\infty$, and $W^{N,0}$ satisfies
Hypothesis~\ref{hyp:graph_initial}.
Then, we have 
\begin{equation}\label{eq:error-estimate}
\sup_{i}\mathbb{E}|X_i^N(t)-\bar X_i^N(t)| +  \sup_{i,j}\mathbb{E}|W_{ij}^N(t)-\bar W_{ij}^N(t)|  \leq   \frac{C(T)}{\sqrt{N}}
\end{equation}
with the constant $C(T)$ being defined as $C(T):= \bigg( 2{C_4(T) M_\phi + C_\Lambda (1+ C_4(T)) } \bigg) \exp(C_5(T)T)$   where $C_4(T) := (C + C_\Lambda T)  \exp(C_\Lambda T)$ and $C_5(T) := \max( 2 C_4(T) L_\phi  +3 L\Lambda, M_\phi  +2 L\Lambda)$.
\end{lemma}
\begin{proof}
We do the difference between the first equations of \eqref{eq:particle_syst_restrictive} and \eqref{eq:intermediate1} and integrate over time to obtain 
$$\begin{array}{lcl}
\dsp \left| X_i^N - \bar X_i^N \right|(t) &  = & \dsp \left| \int_0^t \frac{1}{N} \sum_{j=1}^N \left( W_{ij}^N(s) \phi(s,X_i^N(s),X_j^N(s))  - \E_i \croch{\bar W_{ij}^N(s) \phi(s,\bar X_i^N(s),\bar X_j^N(s))} \right) ds  \right|  \\
 & \leq & \dsp \frac{1}{N} \sum_{j=1}^N  \int_0^t  \left| W_{ij}^N(s) \phi(s,X_i^N(s),X_j^N(s))  - W_{ij}^N(s) \phi(s,\bar X_i^N(s), \bar X_j^N(s))  \right| ds\\
  &  & +  \dsp \frac{1}{N} \sum_{j=1}^N  \int_0^t  \left| W_{ij}^N(s) - \bar W_{ij}^N(s) \right| \left|\phi(s,\bar X_i^N(s), \bar X_j^N(s)) \right| ds \\
  & & \dsp + \int_0^t \left| \frac{1}{N} \sum_{j=1}^N \left( \bar W_{ij}^N(s) \phi(s,\bar X_i^N(s), \bar X_j^N(s))  - \E_i \croch{\bar W_{ij}^N(s) \phi(s,\bar X_i^N(s),\bar X_j^N(s))} \right) \right| ds.
\end{array}$$
Thus, taking the expectation and using Hypothesis \ref{hyp:phi} and equation \eqref{eq:bound_weights1}, we obtain for all $t \in [0,T]$,
$$\begin{array}{lcl}
\dsp \E {\left| X_i^N - \bar X_i^N \right|}(t) & \leq & \dsp \int_0^t  \E \croch{\frac{1}{N} \sum_{j=1}^N   \left| W_{ij}^N(s)\right| L_\phi  \left(|X_i^N - \bar X_i^N |(s) + |X_j^N  - \bar X_j^N|(s)\right)  }ds\\
  &  & +  \dsp   M_\phi \int_0^t \sup_{i,j} \E \left| W_{ij}^N - \bar W_{ij}^N \right|(s) ds \\
  & & \dsp + \int_0^t \E \left| \frac{1}{N} \sum_{j=1}^N \left( \bar W_{ij}^N \phi(s,\bar X_i^N(s), \bar X_j^N(s))  - \E_i \croch{\bar W_{ij}^N \phi(s,\bar X_i^N(s),\bar X_j^N(s))} \right) \right| ds\\
  & \leq & \dsp 2 C_4(T) L_\phi    \int_0^t  \sup_i \E {|X_i^N - \bar X_i^N|(s)}ds\\
  &  & +  \dsp    M_\phi  \int_0^t \sup_{i,j} \E \left| W_{ij}^N - \bar W_{ij}^N \right|(s) ds \\
  & & \dsp + \int_0^t \E \left| \frac{1}{N} \sum_{j=1}^N \left( \bar W_{ij}^N \phi(s,\bar X_i^N(s), \bar X_j^N(s))  - \E_i \croch{\bar W_{ij}^N \phi(s,\bar X_i^N(s),\bar X_j^N(s))} \right) \right| ds
\end{array}$$
with $C_4(T) = (C + C_\Lambda T)  \exp(C_\Lambda T)$. Let us deal with the last term. By Jensen's inequality, we obtain $$ 
\begin{array}{lcl}
A & :=  & \dsp \left( \E \left| \frac{1}{N} \sum_{j=1}^N \left( \bar W_{ij}^N \phi(s,\bar X_i^N, \bar X_j^N)  - \E_i \croch{\bar W_{ij}^N \phi(s,\bar X_i^N,\bar X_j^N)} \right) \right|\right)^2\\
& \leq &  \dsp \E \croch{\left| \frac{1}{N} \sum_{j=1}^N \left( \bar W_{ij}^N \phi(s,\bar X_i^N, \bar X_j^N)  - \E_i \croch{\bar W_{ij}^N \phi(s,\bar X_i^N,\bar X_j^N)} \right) \right|^2}\\
& = &  \dsp \E \left[\frac{1}{N^2} \sum_{j_1=1}^N \sum_{j_2=1}^N \left( \bar W_{ij_1}^N \phi(s,\bar X_i^N, \bar X_{j_1}^N)  - \E_i \croch{\bar W_{ij_1}^N \phi(s,\bar X_i^N,\bar X_{j_1}^N)} \right)  \right.\\
& & \dsp ~~~~~~~~~~~~~~~~~~~~~~ \left( \bar W_{ij_2}^N \phi(s,\bar X_i^N, \bar X_{j_2}^N)  - \E_i \croch{\bar W_{ij_2}^N \phi(s,\bar X_i^N,\bar X_{j_2}^N)} \right)  \Bigg].
\end{array}$$
Moreover, by  the law of total expectation, we know that 
\begin{multline}
\dsp \E \Bigg[\Bigg( \bar W_{ij_1}^N \phi(s,\bar X_i^N, \bar X_{j_1}^N)  - \E_i \croch{\bar W_{ij_1}^N \phi(s,\bar X_i^N,\bar X_{j_1}^N)} \Bigg)  \Bigg( \ \bar W_{ij_2}^N \phi(s,\bar X_i^N, \bar X_{j_2}^N)  - \E_i \croch{\bar W_{ij_2}^N \phi(s,\bar X_i^N,\bar X_{j_2}^N)} \Bigg) \Bigg]\\
\dsp = \E \Bigg[\E_i\Bigg[\Bigg( \bar W_{ij_1}^N \phi(s,\bar X_i^N, \bar X_{j_1}^N)  - \E_i \croch{\bar W_{ij_1}^N \phi(s,\bar X_i^N,\bar X_{j_1}^N)} \Bigg)  \Bigg( \ \bar W_{ij_2}^N \phi(s,\bar X_i^N, \bar X_{j_2}^N)  - \E_i \croch{\bar W_{ij_2}^N \phi(s,\bar X_i^N,\bar X_{j_2}^N)} \Bigg) \Bigg]\Bigg].
\end{multline}
Then, for $j_1 \neq j_2$,  $(X_i,X_{j_1},W_{ij_1})$ and $(X_i,X_{j_2},W_{ij_2})$ are conditionally independent given $\mathcal{F}_i$. Thus, the previous expectation is equal to the product of the expectations, each of them being equal to zero. Thus, we only have left the diagonal terms and using the bound \eqref{eq:bound_weights2}, then  we get 
$$\begin{array}{lcl}
A & \leq   &\dsp  \E \croch{\frac{1}{N^2} \sum_{j=1}^N \left( \bar W_{ij_1}^N \phi(s,\bar X_i^N, \bar X_{j_1}^N)  - \E_i \croch{\bar W_{ij_1}^N \phi(s,\bar X_i^N,\bar X_{j_1}^N)} \right)^2 }\\
& \leq & \dsp \frac{4 M_\phi^2}{N^2} NC_4(T)^2 = \frac{4C_4(T)^2M_\phi^2 }{N}. 
\end{array}$$
Thus, finally, we have 
$$\begin{array}{lcl}
\sup_i \dsp \E {\left| X_i^N - \bar X_i^N \right|}(t) & \leq & \dsp 2 C_4(T) L_\phi    \int_0^t  \sup_i \E {|X_i^N - \bar X_i^N|(s)}ds\\
  &  & +  \dsp    M_\phi  \int_0^t \sup_{i,j} \E \left| W_{ij}^N - \bar W_{ij}^N \right|(s) ds \\
  & & \dsp + \frac{2{C_4(T){T}M_\phi }}{\sqrt{N}}. 
\end{array}$$
Moreover, taking the difference between this time the second equations of \eqref{eq:particle_syst_restrictive} and \eqref{eq:intermediate1}, and integrate over time, we obtain 
$$\begin{array}{lcl}
\dsp \E \left| W_{ij}^N - \bar W_{ij}^N \right|(t) &  = & \dsp\int_0^t  \E \Bigg| \frac{1}{N^2} \sum_{\substack{i_1=1 \\ i_1\neq i}}^N \sum_{\substack{j_1=1 \\ j _1\neq i}}^N\Bigg( \Lambda_{ij}(X_j^N,W_{ij}^N,X_{i_1}^N,X_{j_1}^N,W_{{i_1}{j_1}}^N)(s) \\
 & & ~~~~~~~~~~~~~~~~~~~~~~~~~~~~~~~~~~~- \E_{i,j} \croch{\Lambda_{ij}(\bar X_j^N,\bar W_{ij}^N,\bar X_{i_1}^N,\bar X_{j_1}^N,\bar W_{{i_1}{j_1}}^N)(s)} \Bigg) \Bigg| ds   \\
  &  \leq & \dsp\int_0^t  \E \Bigg| \frac{1}{N^2} \sum_{\substack{i_1=1 \\ i_1\neq i}}^N \sum_{\substack{j_1=1 \\ j _1\neq i}}^N \Bigg( \Lambda_{ij}(X_j^N,W_{ij}^N,X_{i_1}^N,X_{j_1}^N,W_{{i_1}{j_1}}^N)(s) \\
 & & ~~~~~~~~~~~~~~~~~~~~~~~~~~~~~~~~~~~- {\Lambda_{ij}(\bar X_j^N,\bar W_{ij}^N,\bar X_{i_1}^N,\bar X_{j_1}^N,\bar W_{{i_1}{j_1}}^N)(s)} \Bigg) \Bigg| ds   \\
  &   & \dsp + \int_0^t  \E \Bigg| \frac{1}{N^2}\sum_{\substack{i_1=1 \\ i_1\neq i}}^N \sum_{\substack{j_1=1 \\ j _1\neq i}}^N \Bigg( \Lambda_{ij}(\bar X_j^N,\bar W_{ij}^N,\bar X_{i_1}^N,\bar X_{j_1}^N,\bar W_{{i_1}{j_1}}^N)(s) \\
 & & ~~~~~~~~~~~~~~~~~~~~~~~~~~~~~~~~~~~- \E_{i,j} \croch{\Lambda_{ij}(\bar X_j^N,\bar W_{ij}^N,\bar X_{i_1}^N,\bar X_{j_1}^N,\bar W_{{i_1}{j_1}}^N)(s)} \Bigg) \Bigg| ds   \\
   &  \leq & \dsp L_\Lambda \int_0^t \left(3  \sup_i \dsp \E {\left| X_i^N - \bar X_i^N \right|}(s)  + 2   \sup_{i,j} \dsp \E {\left| W_{ij}^N - \bar W_{ij}^N \right|}(s)\right) ds \\
  &   & \dsp + \int_0^t  \E \Bigg| \frac{1}{N^2} \sum_{\substack{i_1=1 \\ i_1\neq i}}^N \sum_{\substack{j_1=1 \\ j _1\neq i}}^N \Bigg( \Lambda_{ij}(\bar X_j^N,\bar W_{ij}^N,\bar X_{i_1}^N,\bar X_{j_1}^N,\bar W_{{i_1}{j_1}}^N)(s) \\
 & & ~~~~~~~~~~~~~~~~~~~~~~~~~~~~~~~~~~~- \E_{i,j} \croch{\Lambda_{ij}(\bar X_j^N,\bar W_{ij}^N,\bar X_{i_1}^N,\bar X_{j_1}^N,\bar W_{{i_1}{j_1}}^N)(s)} \Bigg) \Bigg| ds 
\end{array}$$
where the last inequality has been obtained using Hypothesis  \ref{hyp:psi}. We deal with  the second term and define 
$$\begin{array}{lcl}
B  & :=   & \dsp \Bigg( \E \Bigg| \frac{1}{N^2} \sum_{\substack{i_1=1 \\ i_1\neq i}}^N \sum_{\substack{j_1=1 \\ j _1\neq i}}^N \Bigg( \Lambda_{ij}(\bar X_j^N,\bar W_{ij}^N,\bar X_{i_1}^N,\bar X_{j_1}^N,\bar W_{{i_1}{j_1}}^N)(s) \\
 & & ~~~~~~~~~~~~~~~~~~~~~~~~~~~~~~~~~~~- \E_{i,j} \croch{\Lambda_{ij}(\bar X_j^N,\bar W_{ij}^N,\bar X_{i_1}^N,\bar X_{j_1}^N,\bar W_{{i_1}{j_1}}^N)(s)} \Bigg) \Bigg| \Bigg)^2.
 \end{array}$$
 Applying Jensen's inequality, we obtain the bound 
$$\begin{array}{lcl}
B  & \leq    & \dsp \E \Bigg[\frac{1}{N^4}  \Bigg( \sum_{\substack{i_1=1 \\ i_1\neq i}}^N \sum_{\substack{j_1=1 \\ j _1\neq i}}^N \Bigg( \Lambda_{ij}(\bar X_j^N,\bar W_{ij}^N,\bar X_{i_1}^N,\bar X_{j_1}^N,\bar W_{{i_1}{j_1}}^N)(s) \\
 & & ~~~~~~~~~~~~~~~~~~~~~~~~~~~~~~~~~~~- \E_{i,j} \croch{\Lambda_{ij}(\bar X_j^N,\bar W_{ij}^N,\bar X_{i_1}^N,\bar X_{j_1}^N,\bar W_{{i_1}{j_1}}^N)(s)} \Bigg)  \Bigg)^2  \Bigg]\\
 &=    & \dsp \E \Bigg[\frac{1}{N^4}  \Bigg( \sum_{\substack{i_1, j_1,  i_2,  j_2 \\ \neq i}}\Bigg( \Lambda_{ij}(\bar X_j^N,\bar W_{ij}^N,\bar X_{i_1}^N,\bar X_{j_1}^N,\bar W_{{i_1}{j_1}}^N)(s) - \E_{i,j} \croch{\Lambda_{ij}(\bar X_j^N,\bar W_{ij}^N,\bar X_{i_1}^N,\bar X_{j_1}^N,\bar W_{{i_1}{j_1}}^N)(s)} \Bigg)\\
 & & 
~~~~~~~~~~~~~~~~~~~~ \Bigg( \Lambda_{ij}(\bar X_j^N,\bar W_{ij}^N,\bar X_{i_2}^N,\bar X_{j_2}^N,\bar W_{{i_2}{j_2}}^N)(s) - \E_{i,j} \croch{\Lambda_{ij}(\bar X_j^N,\bar W_{ij}^N,\bar X_{i_2}^N,\bar X_{ j_2}^N,\bar W_{{ i_2}{j_2}}^N)(s)} \Bigg)  \Bigg].
\end{array}$$
By the same reasoning as previously, for $(i_1,j_1) \neq (i_2,j_2)$, conditionally to $\mathcal{F}_{i,j}$, we have $(\bar X_j^N,\bar W_{ij}^N,\bar X_{i_1}^N,\bar X_{j_1}^N,\bar W_{{i_1}{j_1}}^N)$ is independent from $(\bar X_j^N,\bar W_{ij}^N,\bar X_{i_2}^N,\bar X_{ j_2}^N,\bar W_{{ i_2}{j_2}}^N)$ and we only have left the diagonal terms. Therefore, using Hypothesis \ref{hyp:psi} and equation \eqref{eq:bound_weights2}, we obtain 
$$\begin{array}{lcl}
B  & \leq      & \dsp \E \Bigg[\frac{1}{N^4}  \sum_{\substack{i_1=1 \\ i_1\neq i}}^N \sum_{\substack{j_1=1 \\ j _1\neq i}}^N  \Bigg( \Lambda_{ij}(\bar X_j^N,\bar W_{ij}^N,\bar X_{i_1}^N,\bar X_{j_1}^N,\bar W_{{i_1}{j_1}}^N)(s) \\
 & & ~~~~~~~~~~~~~~~~~~~~~~~~~~~~~~~~~~~- \E_{i,j} \croch{\Lambda_{ij}(\bar X_j^N,\bar W_{ij}^N,\bar X_{i_1}^N,\bar X_{j_1}^N,\bar W_{{i_1}{j_1}}^N)(s)} \Bigg)^2  \Bigg]\\
& \leq & \dsp {\frac{{4} C_\Lambda^2}{N^2} (1+ C_4(T))^2}.
\end{array}$$
and we deduce that 
$$\begin{array}{lcl}
\dsp \sup_{i,j} \E \left| W_{ij}^N - \bar W_{ij}^N \right|(t) &  \leq & \dsp L_\Lambda \int_0^t \left( 3  \sup_i \dsp \E {\left| X_i^N - \bar X_i^N \right|}(s)  
+ 2   \sup_{i,j} \dsp \E {\left| W_{ij}^N - \bar W_{ij}^N \right|}(s)\right) ds \\ 
& &\dsp  + \frac{{2}C_\Lambda{T}}{N} (1+ C_4(T)). 
\end{array}$$
Finally, altogether, by Gronwall's lemma, we get 
$$\begin{array}{lcl}
\dsp
\sup_i \dsp \E {\left| X_i^N - \bar X_i^N \right|}(t) + \sup_{i,j} \E \left| W_{ij}^N - \bar W_{ij}^N \right|(t)    & \leq & \dsp (2 C_4(T) L_\phi  +3 L\Lambda)   \int_0^t  \sup_i \E {|X_i^N - \bar X_i^N|(s)}ds\\
  &  & +  \dsp   ( M_\phi  +2 L\Lambda) \int_0^t \sup_{i,j} \E \left| W_{ij}^N - \bar W_{ij}^N \right|(s) ds \\
    & & \dsp + \frac{2{C_4(T) M_\phi }}{\sqrt{N}} +  \frac{C_\Lambda}{N} (1+ C_4(T))\\
    & \leq & \dsp  C_5(T) \int_0^t \left( \sup_i \E {|X_i^N - \bar X_i^N|(s)} +  \sup_{i,j} \E \left| W_{ij}^N - \bar W_{ij}^N \right|(s) \right) ds \\
    & & \dsp + \frac{2{C_4(T) {T}M_\phi + {2} C_\Lambda {T} (1+ C_4(T)) }}{\sqrt{N}}
\end{array}$$
 where $C_5(T) = \max( 2 C_4(T) L_\phi  +3 L\Lambda, M_\phi  +2 L\Lambda)$.
Then, by Gronwall's lemma, we get 
$$\begin{array}{lcl}
\dsp
\sup_i \dsp \E {\left| X_i^N - \bar X_i^N \right|}(t) + \sup_{i,j} \E \left| W_{ij}^N - \bar W_{ij}^N \right|(t)    & \leq & \dsp  \frac{C(T)}{\sqrt{N}} 
\end{array}$$
with $C(T) = \bigg( 2{C_4(T) {T}M_\phi + {2} C_\Lambda {T} (1+ C_4(T)) } \bigg) \exp(C_5(T)T)$.
\end{proof}

\subsection{Comparison of the graphon reformulations}

The main interest of this paper lies in the asymptotic regime $N \to \infty$. As we have seen previously, within the framework of dense graph theory, the natural limiting object that arises is the graphon. It provides a continuum representation by replacing discrete indices with continuous ones, thereby yielding a simpler description, particularly relevant when $N$ becomes large. Thus, following the approach in \cite{JabinPoyatoSoler21,AyiPoyatoPouradierDuteil}, we introduce the graphon reformulation.

\begin{definition}[Graphon reformulation] \label{def:graphon_reformulation}
For every $N \in \mathbb{N}^*$ and $t \in [0,T]$, we define 
$$\mu_t^{N} \in \mathcal{P}_\nu( I^2 \times \mathbb{R}^d \times \mathbb{R}^d \times \mathbb{R}), \, \bar \mu_t^{N} \in \mathcal{P}_\nu( I^2 \times \mathbb{R}^d \times \mathbb{R}^d \times \mathbb{R} )$$ and  $$\Lambda_N: I^2 \times \mathbb{R}^d \times \mathbb{R} \times \mathbb{R}^d  \times \mathbb{R}^d \times \mathbb{R} \to \R$$
where each element is given as follows 
\begin{equation}\label{eq:empirical_measure}
\mu_t^{N,\xi,\zeta}(x,y,w) := \sum_{i=1}^N  \sum_{j=1}^N \mathbf{1}_{I_i^N}(\xi) \mathbf{1}_{I_j^N}(\zeta) \delta_{X_i^N(t)}(x) \delta_{X_j^N(t)}(y) \delta_{W_{ij}^N(t)}(w), \, \xi,\zeta \in I
\end{equation}

 \begin{equation}\label{eq:empirical_measure_intermediate}
\bar \mu_t^{N,\xi,\zeta}(x,y,w) := \sum_{i=1}^N  \sum_{j=1}^N \mathbf{1}_{I_i^N}(\xi) \mathbf{1}_{I_j^N}(\zeta) \lij(x,y,w), \, \xi,\zeta \in I
\end{equation}
 and 
 \begin{equation}\label{eq:empirical_measure_weights}\Lambda_N(\xi,\zeta,y,w, \tx, \ty, \tw)  := \sum_{i=1}^N  \sum_{j=1}^N \mathbf{1}_{I_i^N}(\xi) \mathbf{1}_{I_j^N}(\zeta) \Lambda_{ij}(y,w, \tx, \ty, \tw), \, \xi,\zeta \in I
\end{equation}
\end{definition}
The interest of this graphon reformulation is that it allows us to connect the intermediate particle system with a solution of the Vlasov-type equation.
\begin{lemma} \label{lem:graphon_solution_Vlasov}
Assume that the interaction function satisfy Hypothesis \ref{hyp:phi} and the weight dynamics Hypotheses \ref{hyp:psi} and \ref{hyp:struct_lambda}. 
Let the particle systems \eqref{eq:particle_syst_restrictive} and
\eqref{eq:intermediate1} be initialized with random variables
$(X^{N,0},W^{N,0})$ such that the family
$\{X_i^{N,0},\, W_{ij}^{N,0} : 1\le i,j\le N\}$ is independent,
$\E|X_i^{N,0}|<+\infty$, and $W^{N,0}$ satisfies
Hypothesis~\ref{hyp:graph_initial}. Consider the unique solution $(\bar X^{N}, \bar W^{N})$ solution to \eqref{eq:intermediate1}, the associate laws $$\bar \lambda_t^{N,i,j} (\cdot, \cdot, \cdot) = \text{law} \left((\bxi(t), \bxj(t), \bwij(t))\right)$$ and the graphon reformulation $(\bar \mu^N,\Lambda_N)$ introduced in Definition \ref{def:graphon_reformulation}. Then, $\bar \mu^N$ is a distributional solution to the Vlasov-type equation \eqref{eq:vlasov-equation}-\eqref{eq:vlasov-force} with weight dynamics $\Lambda_N$ and initial datum 
$\bar \mu_{t=0}^{N,\xi,\zeta} := \sum_{i=1}^N  \sum_{j=1}^N \mathbf{1}_{I_i^N}(\xi) \mathbf{1}_{I_j^N}(\zeta) \text{law} \left((X_i^{N,0}, X_j^{N,0}, W_{ij}^{N,0})\right), \, \xi,\zeta \in I$
\end{lemma}
\begin{proof}
 We start by checking that the law $\bar \lambda_t^{N,i,j} (\cdot, \cdot, \cdot) = \text{law} \left((\bxi(t), \bxj(t), \bwij(t))\right)$ satisfy the following equation: 
 
 \begin{multline} \label{eq:law_part}
 \partial_t \lij(x,y,w) + \frac{1}{N}  \sum_{k=1}^N\nabla_x \cdot \left( \left(  \int_{\R^d \times \R}  \tw \phi (t, x, \ty) \int_{\R^d} \bar  \lambda_t^{N,i,k} (\tx, \ty, \tw) d \tx d\ty d\tw \right)  \lij(x,y,w) \right) \\
 + \frac{1}{N}  \sum_{k=1}^N\nabla_y \cdot \left( \left(  \int_{\R^d \times \R}  \tw \phi (t, y, \ty) \int_{\R^d} \bar  \lambda_t^{N,j,k} (\tx, \ty, \tw) d \tx d\ty d\tw \right)
  \lij(x,y,w) \right) \\
  +  \frac{1}{N^2} \sum_{\substack{i_1=1 \\ i_1\neq i}}^N \sum_{\substack{j_1=1 \\ j _1\neq i}}^N   \partial_w \left( \left(  \int_{\R^d \times \R^d \times \R} \Lambda_{ij}(y,w, \tx, \ty, \tw)   \lijun(\tx, \ty, \tw)  d \tx d\ty d\tw \right)  \lij(x, y,w)   \right) = 0.
 \end{multline}
 
 Let $\varphi \in \mathcal{C}_c^1(\R^d\times \R^d \times \R)$,  then we have
$$\begin{array}{l}
\dsp  \frac{d}{dt}  \int_{\R^d \times \R^d \times \R} \varphi(x,y,w)  \lij(x,y,w)  dx dy dw \\
\dsp  =  \frac{d}{dt}  \E \croch{\varphi(\bxi, \bxj, \bwij)}\\
\dsp = \E \croch{\nabla_x \varphi(\bxi, \bxj, \bwij) \cdot  \frac{d}{dt} \bar X_i^N } + \E \croch{\nabla_y \varphi(\bxi, \bxj, \bwij) \cdot  \frac{d}{dt} \bar X_j^N } \\
\dsp  ~~~~~~~~~  ~~  +  \E \croch{\partial_w \varphi(\bxi, \bxj, \bwij)   \frac{d}{dt}  \bar W_{ij}^N}  \\
\dsp = \E \croch{\nabla_x \varphi(\bxi, \bxj, \bwij) \cdot   \frac{1}{N} \sum_{k=1}^N \int_{\R^d \times \R}  \tw \phi (t, \bxi, \ty) \int_{\R^d} \bar  \lambda_t^{N,i,k} (\tx, \ty, \tw) d \tx d\ty d\tw } \\
\dsp + \E \croch{\nabla_y \varphi(\bxi, \bxj, \bwij) \cdot  \frac{1}{N} \sum_{k=1}^N \int_{\R^d \times \R}  \tw \phi (t, \bxj, \ty) \int_{\R^d} \bar  \lambda_t^{N,j,k} (\tx, \ty, \tw) d \tx d\ty d\tw  } \\
\dsp +  \E \left[\partial_w \varphi(\bxi, \bxj, \bwij)    \frac{1}{N^2} \sum_{\substack{i_1=1 \\ i_1\neq i}}^N \sum_{\substack{j_1=1 \\ j _1\neq i}}^N \int_{\R^d \times \R^d \times \R}\ {\Lambda_{ij}(\bar X_j^N,\bar W_{ij}^N, \tx, \ty, \tw)} \right.\ \left.  \lijun( \tx, \ty, \tw) d \tx d\ty d\tw  \right]  \\
 \end{array} $$
 $$\begin{array}{l}
\dsp =\left( \int_{\R^d \times \R^d \times \R}\nabla_x \varphi(x,y,w) \cdot   \frac{1}{N} \sum_{k=1}^N \int_{\R^d \times \R}  \tw \phi (t, x, \ty) \int_{\R^d} \bar  \lambda_t^{N,i,k} (\tx, \ty, \tw) d \tx d\ty d\tw \right.\\
\dsp  \dsp ~~~~~~~~~~~~~~~~~~~~~~~~~ ~~~~~~~~~~~~~~~~~~~~~  \lij(x,y,w) dx dy dw \Bigg) \\
\dsp +\left( \int_{\R^d \times \R^d \times \R}\nabla_y \varphi(x,y,w) \cdot   \frac{1}{N} \sum_{k=1}^N \int_{\R^d \times \R}  \tw \phi (t, y, \ty) \int_{\R^d} \bar  \lambda_t^{N,j,k} ( \tx, \ty, \tw) d \tx d\ty d\tw \right.\\
\dsp  \dsp ~~~~~~~~~~~~~~~~~~~~~~~~~ ~~~~~~~~~~~~~~~~~~~~~  \lij(x,y,w) dx dy dw \Bigg) \\ 
\dsp + \left(\int_{\R^d \times \R^d \times \R}  \partial_w \varphi(x,y,w)    \frac{1}{N^2} \sum_{\substack{i_1=1 \\ i_1\neq i}}^N \sum_{\substack{j_1=1 \\ j _1\neq i}}^N  \int_{\R^d \times \R^d \times \R}\ {\Lambda_{ij}(y,w, \tx, \ty, \tw)}   \lij(x, y,w)  \right. \left.  \lijun(\tx, \ty, \tw) d \tx d\ty d\tw  dx dy dw\right)  
 \end{array} $$
and we recover  that $\bar \lambda_t^{N,i,j}$ satisfies \eqref{eq:law_part} in a weak form. 

By definition, the measure $\bar\mu_t^N$ is piecewise constant with respect to $(\xi,\zeta)$
on the partition $\{I_i^N\times I_j^N\}_{1\le i,j\le N}$, and satisfies
\[
\bar\mu_t^{N,\xi,\zeta} = \bar\lambda_t^{N,i,j}
\quad \text{for } (\xi,\zeta)\in I_i^N\times I_j^N.
\]
Integrating \eqref{eq:law_part} against $\mathbf{1}_{I_i^N}(\xi)\mathbf{1}_{I_j^N}(\zeta)$  and summing over $(i,j)$  yields the following equation 
 \begin{multline} \label{eq:law_part2}
 \partial_t \bar \mu_t^{N,\xi,\zeta}(x,y,w)+ \nabla_x \cdot \left( \left(  \int_{I \times \R^d \times \R}  \tw \phi (t, x, \ty) \int_{\R^d} \bar \mu_t^{N,\xi,\tz}(\tx, \ty, \tw)  d \tx d\ty d\tw d\tz \right)  \bar \mu_t^{N,\xi,\zeta}(x,y,w) \right) \\
 +\nabla_y \cdot \left( \left(  \int_{I \times \R^d \times \R}  \tw \phi (t, y, \ty) \int_{\R^d}   \bar \mu_t^{N,\zeta,\tz}(\tx, \ty, \tw)  d \tx d\ty d\tw d\tz \right) \bar \mu_t^{N,\xi,\zeta}(x,y,w) \right) \\
  +   \partial_w \left( \left(  \int_{ I^2 \times \R^d \times \R^d \times \R}  \Lambda_N(\xi, \zeta,y,w, \tx, \ty, \tw)    \bar \mu_t^{N,\txi,\tz}(\tx, \ty, \tw)  d \tx d\ty d\tw d \txi d \tz  \right)  \bar \mu_t^{N,\xi,\zeta}(x,y,w)   \right) = 0
 \end{multline}
 which concludes the proof.

\end{proof}

We can finally state our result of comparison of the graphon reformulations.
\begin{lemma}[Comparison of the graphon reformulations] \label{lem:comparison_graphons} Suppose that the assumptions of Lemma \ref{lem:graphon_solution_Vlasov} hold. Then, for all $t\in[0,T]$,
\begin{equation}
\begin{array}{lcl}
\dsp d_{1}\!\left(\E \mu_t^{N}, \bar \mu_t^{N}\right)
\le 2 \sup_i \E \bigl| X_i^N(t) - \bar X_i^N(t) \bigr|
   + \sup_{i,j} \E \bigl| W_{ij}^N(t) - \bar W_{ij}^N(t) \bigr|.
\end{array}
\end{equation}
\end{lemma}

\begin{proof}
Fix $(\xi,\zeta)\in I^2$ and let $(i,j)$ be such that $\xi\in I_i^N$ and $\zeta\in I_j^N$.
By definition,
\[
\mu_t^{N,\xi,\zeta}=\delta_{(X_i^N(t),\,X_j^N(t),\,W_{ij}^N(t))},
\qquad
\bar\mu_t^{N,\xi,\zeta}=\bar\lambda_t^{N,i,j}
= \text{law} \left((\bxi(t), \bxj(t), \bwij(t))\right).
\]
Therefore,
\[
\begin{aligned}
d_{\rm BL}\!\left(\E \mu_t^{N,\xi,\zeta}, \bar \mu_t^{N,\xi,\zeta}\right)
&= d_{\rm BL}\!\left(
\E \delta_{(X_i^N(t),X_j^N(t),W_{ij}^N(t))},
\bar\lambda_t^{N,i,j}
\right)\\
&= \sup_{\|\varphi\|_{\rm BL}\le 1}
\int_{\R^d\times\R^d\times\R}
\varphi\, d\!\left(
\E \delta_{(X_i^N(t),X_j^N(t),W_{ij}^N(t))}
- \bar\lambda_t^{N,i,j}
\right)\\
&= \sup_{\|\varphi\|_{\rm BL}\le 1}
\E\!\left[
\varphi(X_i^N(t),X_j^N(t),W_{ij}^N(t))
- \varphi(\bar X_i^N(t),\bar X_j^N(t),\bar W_{ij}^N(t))
\right].
\end{aligned}
\]
Since $\|\varphi\|_{\rm BL}\le 1$ implies $[\varphi]_{\rm Lip}\le 1$, we obtain
\[
d_{\rm BL}\!\left(\E \mu_t^{N,\xi,\zeta}, \bar \mu_t^{N,\xi,\zeta}\right)
\le
\E|X_i^N(t)-\bar X_i^N(t)|
+\E|X_j^N(t)-\bar X_j^N(t)|
+\E|W_{ij}^N(t)-\bar W_{ij}^N(t)|.
\]
Integrating this bound over $(\xi,\zeta)\in I^2$ and using the definition
of the distance $d_1$ yields the desired result.
\end{proof}

\section{Mean-field limit}\label{sec:mfl}
This section is devoted to the statement and proof of the main result of the paper.
\begin{theorem}\label{theo:main}
 Let the interaction function $\phi$  satisfy Hypothesis \ref{hyp:phi}. Let the weight dynamics  $\Lambda$ satisfy Hypotheses \ref{hyp:psi} and \ref{hyp:struct_lambda}, and assume that it does not depend on the variable $x$.  Let the particle systems \eqref{eq:particle_syst_restrictive} be initialized with random variables
$(X^{N,0},W^{N,0})$ such that the family
$\{X_i^{N,0},\, W_{ij}^{N,0} : 1\le i,j\le N\}$ is independent,
$\E|X_i^{N,0}|<+\infty$, and $W^{N,0}$ satisfies
Hypothesis~\ref{hyp:graph_initial}. Consider the unique solution associated $(X^{N}, W^{N})$  and define the measure $\mu^N$ as in Definition \ref{def:graphon_reformulation}.
Suppose that there exists $\mu_0 \in \mathcal{P}_{1,\nu}(I^2 \times \R^d \times \R^d \times \R)$ satisfying  Hypothesis \ref{hyp:supp_compact}  with $\nu$ the Lebesgue measure such that $\dsp \lim_{N \to \infty} d_1( \mu_0^N, \mu_0)  =0$ where $$\mu_0^{N,\xi,\zeta}(x,y,w) := \sum_{i=1}^N  \sum_{j=1}^N \mathbf{1}_{I_i^N}(\xi) \mathbf{1}_{I_j^N}(\zeta)  \lambda^{N,i,j}_0(x,y,w)$$ with $\lambda_0^{N,i,j} ( \cdot, \cdot, \cdot) = \text{law} \left((X_i^{N,0}, X_j^{N,0}, W_{ij}^{N,0})\right)$. Let $\mu \in \mathcal{C}([0,T], \mathcal{P}_{1,\nu}(I^2 \times \mathbb{R}^d \times \mathbb{R}^d \times \mathbb{R}))$ be the unique distributional solution to \eqref{eq:vlasov-equation}-\eqref{eq:vlasov-force} issued at $\mu_0$ with given $\Lambda$. Then, we have $$\dsp \lim_{N \to \infty}  \sup_{t \in [0,T]} d_1(\E \mu_t^N, \mu_t) = 0.$$
\end{theorem}

\begin{proof}
By the triangle inequality,
\[
\sup_{t\in[0,T]} d_1(\E\mu_t^N,\mu_t)
\le
\sup_{t\in[0,T]} d_1(\E\mu_t^N,\bar\mu_t^N)
+
\sup_{t\in[0,T]} d_1(\bar\mu_t^N,\mu_t).
\]
The first term is controlled by Lemma~\ref{lem:comparison_graphons},
while the second one is estimated using the stability result of
Theorem~\ref{theo:stability}. Hence,
\begin{align*}
\sup_{t\in[0,T]} d_1(\E\mu_t^N,\mu_t)
&\le
2 \sup_i \E|X_i^N(t)-\bar X_i^N(t)|
+
\sup_{i,j}\E|W_{ij}^N(t)-\bar W_{ij}^N(t)|\\
&\quad
+ e^{4C_3(T)T}
\Bigg(
d_1(\bar\mu_0^N,\mu_0)
+
C_3(T)\!\int_0^T\!\!\!\int
|\Lambda_N-\Lambda|\; d\bar\mu_s^N\, d\mu_0\, ds
\Bigg).
\end{align*}
By Proposition~\ref{prop:compacite}, the support of $\mu_t^{\xi,\zeta}$
is uniformly contained in $\tilde B_T(R_X,R_Y,R_M,\phi,\Lambda)$,
so that we may insert the cutoff $\chi_T$ in the above integral.
The convergence $\Lambda_N\to\Lambda$ follows from the Lebesgue
differentiation theorem, and the integrand is uniformly dominated by
\[
2C_\Lambda(1+|w|)\chi_T(x,y,w)\chi_T(\tx,\ty,\tw),
\]
which is integrable. The dominated convergence theorem therefore yields
\[
\lim_{N\to\infty}
\sup_{t\in[0,T]} d_1(\E\mu_t^N,\mu_t)=0,
\]
concluding the proof.
\end{proof}

\begin{rem}
In the case of a non-adaptive dynamical network, that is when $\Lambda = 0$,
the weights remain static along the dynamics.
Assuming moreover that the initial positions are independent
and that the weights are prescribed by a deterministic graphon
$w(\xi,\zeta)$, the solution $\mu_t$ admits the decomposition
\[
\mu_t^{\xi,\zeta}(x,y,w)
= \delta_{w(\xi,\zeta)}(dw)\,\mu_t^\xi(dx)\,\mu_t^\zeta(dy).
\]

As a consequence, a direct computation shows that the pushforward measure
$\Pi_{\xi,x}\#\mu_t$, where
\[
\Pi_{\xi,x}(\xi,\zeta,x,y,w) = (\xi,x)
\]
is the projection onto the first and third components, satisfies the classical
Vlasov-type equation for non-exchangeable particle systems:
\[
\partial_t \mu_t^\xi(x)
+ \nabla_x \cdot \left(
    \left(
        \int_{I \times \mathbb{R}^d}
        \phi(t,x,y)\,\mu_t^\zeta(dy)\,d\zeta
    \right)
    \mu_t^\xi(x)
\right)
= 0.
\]
This shows that the present framework naturally extends the classical
graphon-based Vlasov models to adaptive weighted interactions.

\end{rem}

\section{Links between the continuum limit and the mean-field limit} \label{sec:link}
{In this section, we switch to a completely deterministic setting. To emphasize the difference with the previous setting, we drop the capital and from now on denote $(x^N,w^N)$ the solution to the particle system \eqref{eq:particle_syst} where $x^N := (x_i^N)_{i \in \{1, \dots,N\}}$ and $w^N := (w_{iji}^N)_{i,j \in \{1, \dots,N\}}$.}
In what follows, we lift the restrictive assumption on $\Lambda$ introduced in Section \ref{sec:independance} and reinstate its dependence on the spatial variable $x$, in addition to the other variables on which it already depends.

As established in the previous sections, the mean-field limit yields the convergence of the microscopic system \eqref{eq:particle_syst} toward the solution of the Vlasov-type equation \eqref{eq:vlasov-equation}-\eqref{eq:vlasov-force} with weight dynamics $\Lambda$.  When considering the large-population limit, i.e., as $N \to \infty$, an alternative approach, commonly referred to as the continuum limit or graph limit, becomes available. This framework provides pointwise convergence of the solution to \eqref{eq:particle_syst} toward a function solving an integro-differential Euler-type equation. Whereas the connection between the particle system and the Vlasov-type equation is established through the graphon reformulation, in the continuum-limit framework it is obtained by constructing a suitable piecewise-constant function. 

We  introduce the specific construction required for our analysis. We define $x_N \in \mathcal{C}([0,T];L^\infty(I))$ and $w_N \in \mathcal{C}([0,T];L^\infty(I^2))$ as follows:
\begin{equation}\label{eq:step_functions}
\dsp x_N(t,\xi) = \sum_{i=1}^N x_i^{N}(t) \mathbf{1}_{I_i^N}(\xi), \, \dsp w_N(t,\xi,\zeta) = \sum_{i=1}^N\sum_{j=1}^N  w_{ij}^{N}(t) \mathbf{1}_{I_i^N}(\xi) \mathbf{1}_{I_j^N}(\zeta).
\end{equation}
An application of the result obtained in \cite{Throm_2024} yields the convergence to the following graph limit: 
\begin{equation} \label{eq:GL}
\left\{\begin{array}{lcl}
\dsp \partial_t x(t, \xi) &= &\dsp  \int_I w(t,\xi, \tz) \phi(t,x(t,\xi),x(t,\tz)) d\tz\\
\dsp \partial_t w(t, \xi, \zeta) &= &\dsp  \int_{I \times I} \Lambda (\xi, \zeta, x(t,\xi), x(t,\zeta), w(t,\xi,\zeta), x(t, \txi), x(t,\tz), w(t,\txi, \tz)) d\txi d\tz
\end{array} \right.
\end{equation}
 First, we define what we call the ``continuous'' empirical measure. 
\begin{definition}[``Continuous'' empirical measure] \label{def:cont_meas_emp} Let the interaction function $\phi$  satisfy Hypothesis \ref{hyp:phi}. Let the weight dynamics $\Lambda$ satisfy Hypotheses \ref{hyp:psi} and \ref{hyp:struct_lambda}. Let $x_0 \in   L^\infty(I;\R^d)$, $ w_0 \in   L^\infty(I^2)$ and $(x,w) \in \mathcal{C}([0,T];L^\infty(I;\R^d)) \times  \mathcal{C}([0,T];L^\infty(I^2)) $ be the the solutions to the graph limit equation \eqref{eq:GL} with weight dynamics $\Lambda$ and initial conditions given by $x(0, \cdot) = x_0$ and $w(0, \cdot, \cdot) = w_0$.  We define the ``continuous'' empirical measure
 as follows:
\begin{equation} \label{eq:cont_emp_meas}
\tilde \mu_t^{\xi,\zeta}(x,y,w) := \delta_{x(t,\xi)}(x) \delta_{x(t,\zeta)}(y) \delta_{w(t,\xi,\zeta)}(w).
\end{equation}
\end{definition}
As explained in our review paper \cite{AyiPouradierDuteil24}, we can actually connect the Vlasov   equation to the graph limit equation. We proceed analogously in our more complex case.
\begin{prop} \label{prop:tilde_vlasov}
Let $\tilde \mu_t$ be the measure defined in Definition \ref{def:cont_meas_emp}. Then, $\tilde \mu_t$ satisfies the Vlasov-type equation \eqref{eq:vlasov-equation}-\eqref{eq:vlasov-force}.
\end{prop}
\begin{proof}
For $\varphi \in \mathcal{C}_c^1(I^2 \times \R^d\times \R^d \times \R)$, we get 
\begin{equation}
\begin{array}{l}
\dsp \frac{d}{dt} \int_{I^2 \times \R^d \times \R^d \times \R} \varphi (\xi,\zeta,x,y,w) \tilde \mu_t^{\xi,\zeta}(dx,dy,dw) d\xi d \zeta \\
\dsp =  \frac{d}{dt}  \int_{I^2} \varphi (\xi,\zeta, x(t,\xi), x(t,\zeta), w(t,\xi,\zeta)) d\xi d\zeta \\
\dsp = \int_{I^2}  \nabla_x \varphi (\xi,\zeta, x(t,\xi), x(t,\zeta), w(t,\xi,\zeta)) \cdot   \left( \int_I w(t ,\xi, \tz) \phi(t,x(t,\xi),x(t,\tz)) d\tz\right) d\xi d\zeta \\
+ \dsp \int_{I^2}  \nabla_y \varphi (\xi,\zeta, x(t,\xi), x(t,\zeta), w(t,\xi,\zeta)) \cdot   \left( \int_I w(t,\zeta, \tz) \phi(t,x(t,\zeta),x(t,\tz)) d\tz\right) d\xi d\zeta \\
+ \dsp \int_{I^2}   \partial_w \varphi (\xi,\zeta, x(t,\xi), x(t,\zeta), w(t,\xi,\zeta)) \cdot  \left( \int_{I^2} \Lambda (\xi, \zeta, x(t,\xi), x(t,\zeta), w(t,\xi,\zeta), x(t, \txi), x(t,\tz), w(t,\txi, \tz))  d\txi d\tz \right) d\xi d \zeta \\
\dsp = \int_{I^2\times \R^d \times \R^d \times \R}  \nabla_x \varphi (\xi,\zeta, x,y,w) \cdot   \left( \int_{I \times \R^d \times \R^d \times \R}  \tw \phi(t,x,\ty)  \tilde \mu_t^{\xi,\tz}(d\tx,d\ty,d\tw)   d\tz \right) \tilde \mu_t^{\xi,\zeta}(dx,dy,dw) d\xi d\zeta \\
\dsp +  \int_{I^2\times \R^d \times \R^d \times \R}  \nabla_y \varphi (\xi,\zeta, x,y,w) \cdot   \left( \int_{I \times \R^d \times \R^d \times \R}  \tw \phi(t,y,\ty)  \tilde \mu_t^{\zeta,\tz}(d\tx,d\ty,d\tw)   d\tz \right) \tilde\mu_t^{\xi,\zeta}(dx,dy,dw) d\xi d\zeta \\
+ \dsp \int_{I^2 \times \R^d \times \R^d \times \R}   \partial_w \varphi (\xi,\zeta, x,y,w) \cdot  \left( \int_{I^2\times \R^d \times \R^d \times \R} \Lambda (\xi,\zeta, x,y,w, \tx,\ty,\tw) 
\tilde \mu_t^{\txi,\tz}(d\tx,d\ty,d\tw)  d\txi d\tz \right) \tilde \mu_t^{\xi,\zeta}(dx,dy,dw) d\xi d \zeta.
\end{array}
\end{equation}
Thus, we recover equation \eqref{eq:vlasov-equation}.
\end{proof}

In order to make the connection with the particle system, we recall the definition of the graphon reformulation
\begin{equation*}
\mu_t^{N,\xi,\zeta}(x,y,w) := \sum_{i=1}^N  \sum_{j=1}^N \mathbf{1}_{I_i^N}(\xi) \mathbf{1}_{I_j^N}(\zeta) \delta_{X_i^N(t)}(x) \delta_{X_j^N(t)}(y) \delta_{W_{ij}^N(t)}(w)
\end{equation*}
and we introduce the following measure 
\begin{equation}
\tilde \mu_t^{N,\xi,\zeta}(x,y,w) := \delta_{x_N(t,\xi)}(x) \delta_{x_N(t,\zeta)}(y) \delta_{w_N(\xi,\zeta)(t)}(w)
\end{equation} 
where  $\dsp x_N $ and $w_N$ have been defined in \eqref{eq:step_functions}. By a straightforward computation, we start by noticing that, for all test function, $\varphi \in \mathcal{C}^1_c(I^2 \times  \R^d \times \R^d \times \R)$, we have 
\begin{equation}\label{eq:equivmeas}
\int_{I^2 \times \R^d \times \R^d \times \R} \varphi(\xi,\zeta,x,y,w)   \mu_t^{N,\xi,\zeta}(dx,dy,dw)d\xi d\zeta = \int_{I^2 \times \R^d \times \R^d \times \R} \varphi(\xi,\zeta,x,y,w)  \tilde \mu_t^{N,\xi,\zeta}(dx,dy,dw)d\xi d\zeta.
\end{equation}
We can now compare the different measures and obtain the following result. 
\begin{theorem}\label{prop:convmu}
 Let the interaction function $\phi$  satisfy Hypothesis \ref{hyp:phi}. Let the weight dynamics $\Lambda$ satisfy Hypotheses \ref{hyp:psi} and \ref{hyp:struct_lambda}. Let $x_0 \in   L^\infty(I;\R^d)$, $ w_0 \in   L^\infty(I^2)$ and $(x,w)$ be the solution to the integro-differential system \eqref{eq:GL} with weight dynamics $\Lambda$ and initial conditions given by $x(0, \cdot) = x_0$ and $w(0, \cdot, \cdot) = w_0$. Let $(x^N,w^N) \in\mathcal{C}([0,T]; L^\infty(I;\R^d))  \times \mathcal{C}([0,T]; L^\infty(I^2) )$ be the  solution of the particle system \eqref{eq:particle_syst} with  initial conditions
\begin{equation}
\begin{cases}
\displaystyle x^{N,0} = \left( N \int_{I_i^N} x_0(\xi) d\xi \right)_{i \in \{1, \dots, N \}} \in (\R^d)^N \\
\displaystyle  w^{N,0} = \left( N \int_{I_i^N \times I_j^N} w_0(\xi,\zeta) d\xi d\zeta \right)_{i \in \{1, \dots, N \}} \in (\R)^N.
\end{cases}
\end{equation} 
Let $\mu^N$ be the graphon reformulation defined in \eqref{eq:empirical_measure} and $\tilde \mu$ be the ``continuous'' empirical measure defined in \eqref{eq:cont_emp_meas}. 
Then,  the empirical measure $\mu^N$  converges weakly to  the ``continuous'' empirical measure $\tilde{\mu}$, $\tilde{\mu}$ being a solution to the Vlasov-type equation \eqref{eq:vlasov-equation}-\eqref{eq:vlasov-force}. More precisely, for all test function $\varphi \in \mathcal{C}^1_c(I^2 \times  \R^d \times \R^d \times \R)$ and all $t\in [0,T]$, it holds 
\begin{equation*}
\lim_{N\rightarrow\infty} \int_{I^2 \times \R^d \times \R^d \times \R} \varphi(\xi,\zeta,x,y,w) (  \mu_t^{N,\xi,\zeta}(dx,dy,dw) - \tilde \mu_t^{\xi,\zeta}(dx,dy,dw))d\xi d\zeta  =0.
%\int_{\R^d} \varphi(x) d\mu^N_t(x) = \int_{\R^d} \varphi(x) d\tilde{\mu}_t(x) .
\end{equation*}
\end{theorem}

\begin{proof}
Let $\varphi\in \mathcal{C}_c^1(I^2\times \R^d\times \R^d\times \R)$.
By \eqref{eq:equivmeas} and by the definitions of $\tilde\mu_t^{N}$ and $\tilde\mu_t$,
\[
\int \varphi\, d\mu_t^{N} - \int \varphi\, d\tilde\mu_t
=
\int_{I^2}\Bigl[
\varphi\!\bigl(\xi,\zeta,x_N(t,\xi),x_N(t,\zeta),w_N(t,\xi,\zeta)\bigr)
-
\varphi\!\bigl(\xi,\zeta,x(t,\xi),x(t,\zeta),w(t,\xi,\zeta)\bigr)
\Bigr]\,d\xi\,d\zeta .
\]
Since $\varphi$ is $\mathcal{C}^1$ and compactly supported in $(x,y,w)$, it is Lipschitz in these variables and
\[
\Bigl|\varphi(\xi,\zeta,x_1,y_1,w_1)-\varphi(\xi,\zeta,x_2,y_2,w_2)\Bigr|
\le \|\nabla_{x,y,w}\varphi\|_{L^\infty}\bigl(|x_1-x_2|+|y_1-y_2|+|w_1-w_2|\bigr).
\]
Therefore,
\begin{align*}
\left|\int \varphi\, d\mu_t^{N} - \int \varphi\, d\tilde\mu_t\right|
&\le \|\nabla_{x,y,w}\varphi\|_{L^\infty}
\int_{I^2}\bigl(|x_N(t,\xi)-x(t,\xi)|+|x_N(t,\zeta)-x(t,\zeta)|
+|w_N(t,\xi,\zeta)-w(t,\xi,\zeta)|\bigr)\,d\xi\,d\zeta \\
&\le \|\nabla_{x,y,w}\varphi\|_{L^\infty}
\Bigl(2\|x_N(t,\cdot)-x(t,\cdot)\|_{L^2(I)}
+\|w_N(t,\cdot,\cdot)-w(t,\cdot,\cdot)\|_{L^2(I^2)}\Bigr),
\end{align*}
where we used Cauchy--Schwarz and $|I|=|I^2|=1$.
Taking the supremum over $t\in[0,T]$ and using the graph limit result of \cite{Throm_2024}
(which yields $\|x_N-x\|_{C([0,T];L^2(I;\R^d))}+\|w_N-w\|_{C([0,T];L^2(I^2))}\to 0$ as $N$ goes to $\infty$),
we obtain the desired convergence. Finally, Proposition~\ref{prop:tilde_vlasov} ensures that
$\tilde\mu$ solves \eqref{eq:vlasov-equation}--\eqref{eq:vlasov-force}.
\end{proof}

\begin{rem}
This approach thus yields an alternative proof of the derivation of the Vlasov-type equation, relying on the graph limit as an intermediate step. It should be emphasized that the analysis is restricted to a deterministic setting. On the other hand, the additional restrictive assumption on $\Lambda$,  namely, the absence of dependence on the spatial variable $x$, is no longer required.
\end{rem}

\paragraph*{Acknowledgements} This material is based upon work supported by the National Science Foundation under Grant No. DMS-2424139, while the author was in residence at the Simons Laufer Mathematical Sciences Institute in Berkeley, California, during the Fall 2025 semester.

\bibliography{Biblio_MFL}

\end{document}